\documentclass[a4paper,12pt]{article}
\usepackage[latin1]{inputenc}
\usepackage{amsthm}
\usepackage{amssymb}
\usepackage{amsfonts}
\usepackage[dvips]{graphicx}
\usepackage{fancyhdr}
\usepackage{latexsym}
\usepackage{amsmath}
\usepackage{color}
\usepackage{stmaryrd}

\newtheorem{lem}{Lemma}
\newtheorem{thm}{Theorem}
\newtheorem{defn}{Definition}

\newtheorem{oss}{Remark}

\newcommand{\uvec}{\boldsymbol{u}}

\newcommand{\hvec}{\boldsymbol{h}}

\newcommand{\nvec}{\boldsymbol{n}}
\newcommand{\wvec}{\boldsymbol{w}}

\newcommand{\zvec}{\boldsymbol{z}}

\newcommand{\rhot}{\widetilde{\rho}}
\newcommand{\Jvect}{\widetilde{\boldsymbol{J}}}

\newcommand{\e}{\epsilon}
\newcommand{\Fe}{F_{\epsilon}}
\newcommand{\Fie}{F_{1\epsilon}}
\newcommand{\phie}{\varphi_{\epsilon}}
\newcommand{\phid}{\varphi_{\delta}}
\newcommand{\ue}{\boldsymbol{u}_{\epsilon}}
\newcommand{\ud}{\boldsymbol{u}_{\delta}}
\newcommand{\mue}{\mu_{\epsilon}}
\newcommand{\mud}{\mu_{\delta}}

\allowdisplaybreaks[4]

\numberwithin{equation}{section}

\begin{document}

\title{\bf Global existence of weak solutions for a nonlocal
model for two-phase flows of incompressible fluids with unmatched densities}

\author{
Sergio Frigeri\footnote{Weierstrass Institute for Applied
Analysis and Stochastics, Mohrenstrasse 39, D-10117 Berlin,
Germany, E-mail {\tt  frigeri@wias-berlin.de}}
\newline
}

\maketitle

\vspace{-.4cm}

\noindent {\bf Abstract.} We consider a diffuse interface model for an incompressible isothermal
mixture of two viscous Newtonian fluids with different densities in a bounded domain
in two or three space dimensions. The model is
the nonlocal version of the one recently derived by Abels, Garcke and Gr\"{u}n and
consists in a Navier-Stokes type system 
coupled
with a convective nonlocal Cahn-Hilliard equation.
The density of the mixture depends on an order parameter.
For this nonlocal system we prove existence of global dissipative weak solutions
for the case of singular double-well potentials and non degenerate mobilities.
To this goal
we devise an approach which
is completely independent of the one employed by Abels, Depner and Garcke
to establish existence of weak solutions for the local Abels et al. model.


\vspace{.4cm}


\noindent
{\bf AMS (MOS) subject clas\-si\-fi\-ca\-tion:}
76T99, 35Q30, 35Q35, 76D03, 76D03, 76D05, 76D27.

\section{Introduction}
In this paper we  study the following nonlocal Cahn-Hilliard/Navier-Stokes type
system
\begin{align}
&(\rho\uvec)_t+\mbox{div}(\rho\uvec\otimes\uvec)-2\mbox{div}\big(\nu(\varphi)D\uvec\big)+\nabla\pi+\mbox{div}(\uvec\otimes\Jvect)
=\mu\nabla\varphi+\hvec,\label{Pbor1}\\
&\mbox{div}(\uvec)=0,\label{Pbor2}\\
&\varphi_t+\uvec\cdot\nabla\varphi=\mbox{div}(m(\varphi)\nabla\mu),\label{Pbor3}\\
&\mu=a\varphi-J\ast\varphi+F'(\varphi),\label{Pbor4}\\
&\Jvect=-\beta m(\varphi)\nabla\mu,\qquad\beta=(\widetilde{\rho}_2-\widetilde{\rho}_1)/2,\label{Pbor5}\\
&\rho(\varphi)=\frac{1}{2}(\widetilde{\rho}_2+\widetilde{\rho}_1)+\frac{1}{2}(\widetilde{\rho}_2-\widetilde{\rho}_1)\varphi,\label{Pbor6}
\end{align}
in $Q:=\Omega\times(0,T)$, where $\Omega\subset\mathbb{R}^d$, $d=2,3$, is a bounded smooth domain
and $T>0$ is an arbitrary final time.
The associated boundary and initial conditions are
\begin{align}
&\uvec=0,\qquad\frac{\partial\mu}{\partial\nvec}=0,\qquad\mbox{on }\partial\Omega,\label{Pbor7}\\
&\uvec(0)=\uvec_0,\qquad\varphi(0)=\varphi_0,\qquad\mbox{in }\Omega,\label{Pbor8}
\end{align}
where $\partial\Omega$ is the boundary of $\Omega$ and $\nvec$ is its outward unit normal.

System \eqref{Pbor1}--\eqref{Pbor8} couples a momentum balance equation \eqref{Pbor1}
for the velocity field $\uvec$ with a nonlocal convective Cahn-Hilliard equation \eqref{Pbor3} for
the order parameter $\varphi$ (difference of the volume fractions of the fluids) and
describes the flow and phase separation of an isothermal mixture
of two incompressible Newtonian viscous immiscible fluids with different densities taking into account long-range
interactions between the molecules.
Equation \eqref{Pbor2} accounts for the incompressibility of the mixture,
$\widetilde{\rho}_1, \widetilde{\rho}_2>0$ are the specific constant mass densities of the unmixed fluids,
$\rho=\rho(\varphi)$ given by \eqref{Pbor6} is the density of the mixture,
$\pi$ is the pressure, $\hvec$ is the
external volume force density and $D$ denotes the symmetric gradient, which is defined by $D\uvec:=(\nabla\uvec+\nabla^T\uvec)/2$.
Moreover, if $\boldsymbol{a},\boldsymbol{b}\in\mathbb{R}^d$, we denote by
$\boldsymbol{a}\otimes\boldsymbol{b}$ the tensor defined by $(\boldsymbol{a}\otimes\boldsymbol{b})_{i,j}=a_i b_j$,
for $i,j=1,\cdots,d$.

The mobility $m$ in \eqref{Pbor3} and the viscosity $\nu$ in \eqref{Pbor1} are assumed to be $\varphi-$dependent
and non degenerate, namely both are bounded from below (and above) by positive constants.
The chemical potential $\mu$ contains the spatial convolution $J\ast\varphi$ over $\Omega$, defined by
\begin{align*}
& (J\ast\varphi)(x):=\int_\Omega J(x-y)\varphi(y)dy,\qquad x\in\Omega,
\end{align*}
of the order parameter $\varphi$ with a sufficiently smooth interaction kernel $J$ satisfying $J(z)=J(-z)$.
Moreover, $a$ is given by 
\begin{align*}
& a(x):=\int_\Omega J(x-y)dy,
\end{align*}
for $x\in\Omega$.
The double-well potential $F$ is assumed to be singular and, in particular,
 a physically interesting case that will be included in our analysis
 is the following (see \cite{CH})
\begin{align}
&F(s) = \frac{\theta}{2}((1+s)\log(1+s)+(1-s)\log(1-s)) -\frac{\theta_c}{2}s^2,\qquad 0<\theta<\theta_c,
\label{log}
\end{align}
where $\theta$, $\theta_c$ are the (absolute) temperature and the critical temperature, respectively.

System
\eqref{Pbor1}--\eqref{Pbor6} represents the nonlocal version of the
well known thermodynamically consistent diffuse interface model
for two-phase flow with different densities derived by Abels, Garcke and Gr\"{u}n in \cite{AbGG}.
We recall that the local
model deduced in \cite{AbGG} consists in the above system
with the chemical potential $\mu$ replaced with the local one
\begin{align}
&\mu=-\Delta\varphi+F'(\varphi),\label{locchpot}
\end{align}
and completed with an additional homogeneous Neumann boundary condition for $\varphi$.

We recall that the local chemical potential \eqref{locchpot}
is the first variation of the local free energy functional (see \cite{CH})
$$\mathcal{E}_{loc}(\varphi):=\int_\Omega\Big(\frac{1}{2}|\nabla\varphi|^2+F(\varphi)\Big).$$
Actually, the local free energy considered in \cite{AbGG} contains also a positive
coefficient $a(\varphi)$ multiplying the $|\nabla\varphi|^2$ under the integral. However, here
we have set $a(\varphi)=1$, since this coefficient does not introduce
substantial complications into the analysis.

A different form of the free energy can be associated to the fluid mixture, more precisely
the one proposed in \cite{GL1,GL2} and rigorously justified as a macroscopic limit
of microscopic phase segregation models with particles conserving dynamics (see also \cite{CF}).
In this case the gradient term is replaced by a nonlocal spatial interaction integral, namely
\begin{align*}
&E(\varphi)=\frac{1}{4}\int_\Omega\int_\Omega J(x-y)\big(\varphi(x)-\varphi(y)\big)^2 dxdy+\int_\Omega F(\varphi),
\end{align*}
and the nonlocal chemical potential given by \eqref{Pbor4} is obtained by taking the first variation of $E$.
The physical relevance of nonlocal interactions was already pointed out in the pioneering
paper \cite{Ro} (see also \cite[4.2]{Em} and references therein) and studied
(in the case of constant velocity) for different kind of evolution equations,
mainly Cahn-Hilliard and phase field systems, see, e.g., \cite{BH1, CKRS, GZ, GL1, GL2, GLM, LP,LP2,KRS, KRS2, GG4}.

Diffuse interface models for two-phase flow of fluids with identical densities are very well
established and studied in literature.
These models are based on the so-called model H (see \cite{HH,GPV}, cf. also \cite{Do,M} and references therein),
in which the sharp interface separating the two fluids is replaced by a diffuse one by introducing
an order parameter (cf. \cite{AMW}). They consist of the Navier-Stokes equations for the velocity
field $\uvec$ nonlinearly coupled with a convective Cahn-Hilliard equation for an order parameter
$\varphi$ (cf., for instance, \cite{AMW,GPV,HMR,HH,JV,Kim2012,LMM}).

As far as analytical results for the matched density case (i.e., $\widetilde{\rho}_1=\widetilde{\rho}_2$)
are concerned,
the local Cahn-Hilliard/Navier-Stokes system has been tackled by several authors
(see, e.g., \cite{A1,A2,Abels2,B,CG,GG1,GG2,GG3,HHK,LS,S,ZWH,ZF}
and also \cite{ADT,Bos,GP,KCR} for models with shear dependent viscosity), while the investigation
of its nonlocal version (from the analytical viewpoint concerning well-posedness and long-term behavior)
has started only more recently
(cf., e.g., \cite{CFG,FGG,FG1,FG2,FGK,FGR,FRS}).
In particular, the following situations have been addressed:
regular potential $F$ associated with constant mobility in \cite{CFG,FGG,FG1,FGK}; singular potential associated
with constant mobility in \cite{FG2}; singular potential and degenerate mobility in \cite{FGR}; the case of nonconstant viscosity in \cite{FGG},
which is particularly delicate as far as regularity results in two dimensions are concerned.
In the two-dimensional case it was shown in \cite{FGK} that
for regular potentials and constant mobilities the problem \eqref{Pbor1}--\eqref{Pbor8}
 with $\widetilde{\rho}_1=\widetilde{\rho}_2$ admits
a unique strong solution. Recently, uniqueness was proved also for weak solutions
(see \cite{FGG}). Moreover, relying on the uniqueness results of \cite{FGK} and \cite{FGG}
a related optimal control problem was studied in \cite{FRS} for the case of constant mobility
and regular potential.

Despite the considerable amount of contributions dealing with the matched density case,
analytical results related to models for two-phase flow of fluids with unmatched
densities are quite sporadic. In particular, as far as the local
Abels-Garcke-Gr\"{u}n model is concerned, the first results on existence of weak
solutions were obtained by
Abels, Depner and Garcke in \cite{ADG1}, for the system
with singular potential and non degenerate mobility
and in \cite{ADG2}, for the case of a regular potential and degenerate mobility.
Regarding other diffuse interface models for fluids with different densities
we recall the one considered by Boyer in \cite{B2}. He proved existence of local in time strong
solutions and existence of global weak solutions provided the densities of the fluids are sufficiently close.
We also recall the quasi-incompressible model of Lowengrub and Truskinovsky \cite{LT}, where
the velocity field is not divergence free, for which the first analytical results
were obtained in \cite{Abels2,A3}.

As far as {\itshape nonlocal} models for fluids with unmatched densities
are concerned,
to the best of our knowledge no analytical results have been established so far
and this paper aims to be a first contribution in this direction.
More precisely, the goal of this paper is to prove existence of global
dissipative weak solutions for the nonlocal Abels-Garcke-Gr\"{u}n model
given by system \eqref{Pbor1}--\eqref{Pbor8}, assuming, as in \cite{ADG1},
that the potential is singular and the mobility is non degenerate.
By weak solutions here we mean solutions with the minimum regularity requirement
to allow a finite energy and the validity of an energy dissipation inequality.

Before explaining the strategy of the proof, let us briefly recall
the approach used in \cite{ADG1} 
 and discuss on the possibility
to apply this approach to prove existence of weak solutions for
the nonlocal system \eqref{Pbor1}--\eqref{Pbor8}.
In \cite{ADG1} existence of a weak solution is established by employing an implicit time discretization
scheme. In particular, we point out that the Leray-Schauder fixed-point argument devised for the existence of a solution of the
time-discrete problem (cf. \cite[Lemma 4.3]{ADG1}) relies on the possibility of inverting the local relation
between the chemical potential $\mu$ and $\varphi$ given by \eqref{locchpot}.
This possibility is due to the fact that the relation between $\mu$ and $\varphi$ can be expressed by means
of a maximal monotone operator since $\mu$ can be viewed as the subdifferential of the lower-semicontinuous
convex (up to a quadratic perturbation) local functional $\mathcal{E}_{loc}$.
This approach allows in particular to keep $\varphi$ between the singular points $-1$ and $1$ in all
the analysis. Indeed, the Abels-Garcke-Gr\"{u}n model
is meaningful only when we have a bound on $\varphi$
between $-1$ and $1$ which allows to keep the density $\rho$
bounded from below and above by positive constants.
This bound is ensured, in the case of the local Abels-Garcke-Gr\"{u}n model,
by working with a singular potential as in \cite{ADG1} or with a degenerate mobility as in \cite{ADG2}.

Now, the direct application of the approach devised in \cite{ADG1}
seems hard in the present situation.
Indeed,
the nonlocal chemical potential $\mu$
can no longer be expressed as a subdifferential of a lower semicontinuous convex functional
and therefore the theory of maximal monotone operators is not
directly applicable in the analysis. Moreover, the inversion
of the nonlocal relation between $\mu$ and $\varphi$, under some reasonable
conditions on the kernel $J$ and on the potential $F$, seems to be a rather difficult task.

A possibility to still exploit the approach of \cite{ADG1} in order to prove existence of weak
solutions for the nonlocal system could be to introduce
a local perturbation term of the form $-\delta\Delta\varphi$ on the right hand side of \eqref{Pbor4}.
Existence of a weak solution to the corresponding perturbed system could be
proven, for every $\delta>0$, by suitably adapting the argument of \cite{ADG1}.
Then, existence of a weak solution to the original problem would be
obtained by passing to the limit as $\delta\to 0$ (arguing as in Step III of the proof of the main result
of the present paper).

Although this approach would be possible, however, we propose here an alternative strategy which does not rely at all
on the result of \cite{ADG1}.
Our approach does not employ a time-discretization scheme and does not make use of Leray-Schauder fixed point arguments,
but it is essentially based on the Faedo-Galerkin method
and hence
it is particularly suitable for a possible numerical implementation.



Let us now describe the main lines of our approach.
The starting idea consists in
approximating the singular
potential $F$ by a suitable family of regular potentials $F_\e$ defined on the whole of
$\mathbb{R}$.
This idea, that we already used in \cite{FG2} for the same nonlocal system with matched desities,
is quite classical (see, e.g., \cite{EG,B}). Nevertheless, it leads to some troubles
when applied to our problem. Indeed, if $F$ is replaced by $F_\e$, we shall have
to solve a problem in which the values of $\varphi_\e$ (the $\varphi-$component of the solution
to the approximate problem with potential $F_\e$) are no longer restricted to $(-1,1)$
but belong to the whole of $\mathbb{R}$. This implies that $\rho(\varphi_\e)$
in this $\e-$approximate problem is no longer a-priori bounded from below
by a positive constant and consequently we are in trouble
to get an $L^\infty(L^2)$
estimate for the velocity field $\uvec_\e$.

To overcome this difficulty a possibility is to replace the linear density function $\rho(\varphi)$ by
a fixed smooth extension $\rhot(\varphi)$ from $[-1,1]$ onto $\mathbb{R}$  satisfying
\begin{align}
&0<\rho_\ast\leq\rhot(s)\leq\rho^\ast,\qquad|\rhot^{(k)}(s)|\leq R_k,\qquad\forall s\in\mathbb{R},\quad k=1,2,
\label{bdrho1}\\
&\rhot(s)=\rho(s),\qquad\forall s\in [-1,1],\nonumber
\end{align}
where $\rho_\ast,\rho^\ast,R_1,R_2$ are some given positive constants.
However, we are now led
to a further difficulty. Indeed,
 if we deduce (formally) an energy equation
 from system
 \eqref{Pbor1}--\eqref{Pbor5} in which the linear function $\rho$ is replaced by the nonlinear
 function $\rhot$, by multiplying \eqref{Pbor1} by $\uvec$, \eqref{Pbor3} by $\mu$,
 integrating over $\Omega$ by parts and taking \eqref{Pbor4}--\eqref{Pbor6},
 the incompressibility condition \eqref{Pbor2} and the boundary conditions into account, after some computations we obtain
\begin{align}
&\frac{d}{dt}\Big(\int_\Omega\frac{1}{2}\rhot(\varphi)\uvec^2+E_\e(\varphi)\Big)
+2\int_\Omega\nu(\varphi)|D\uvec|^2+\int_\Omega m(\varphi)|\nabla\mu|^2\nonumber\\
&=\frac{1}{2}\int_\Omega\rhot''(\varphi) m(\varphi)(\nabla\varphi\cdot\nabla\mu) \uvec^2
+\int_\Omega \hvec\cdot\uvec,
\label{enbaldes}
\end{align}
where
\begin{align*}
&E_\e(\varphi)
=\frac{1}{4}\int_\Omega\int_\Omega J(x-y)\big(\varphi(x)-\varphi(y)\big)^2 dxdy+\int_\Omega F_\e(\varphi),
\end{align*}
and
where we have denoted $\uvec_\e,\varphi_\e,\mu_\e$ simply by $\uvec,\varphi,\mu$, for the sake of simplicity.
Therefore, a nonlinear $\rhot$ in system \eqref{Pbor1}--\eqref{Pbor5} destroys the energy balance.
A possibility to handle the nasty nonlinear term on the right hand side of \eqref{enbaldes}
is to recover the energy balance
by inserting, in the approximate problem with the
potential $F_\e$, the term $(1/2)\rhot''(\varphi) m(\varphi)(\nabla\varphi\cdot\nabla\mu) \uvec$
on the left hand side of the momentum-balance equation \eqref{Pbor1}.
This easy trick leads however to still another problem, namely, the problem to pass to the limit
in this new ``artificial" nonlinear term.
The idea at this point is to introduce some suitable regularizing terms in the system,
depending on another positive parameter $\delta$ which will be made go to zero in a second time.
These regularizing terms, which allow to gain enough compactness to be able to pass to the limit,
must be cleverly devised, since: (i) the energy balance should not
be destroyed and (ii) when passing to the limit, firstly as $\e\to 0$ and secondly as $\delta\to 0$,
it should still be possible to prove that the limit $\varphi$ satisfies the bound $|\varphi|<1$.
This bound on $\varphi$  will in particular 
allow the nasty artificial term to vanish in the limit,
permitting then to recover the original momentum-balance equation.
More precisely, the regularizing terms that have been proven to be effective to our purpose
are the term $\delta A^3\uvec$ in \eqref{Pbor1} (this means, more exactly, that the term $\delta(A^{3/2}\uvec,A^{3/2}\wvec)$
is introduced in the variational formulation of \eqref{Pbor1} with test function $\wvec\in D(A^{3/2})$, cf. Definition \ref{wsdefn2};
here $A$ is the Stokes operator
with no-slip boundary condition), and still the term $-\delta\Delta\varphi$ in the expression
of the chemical potential $\mu$.

Summing up, our approach consists in proving existence of a weak solution
to problem \eqref{Pbor1}--\eqref{Pbor6} by approximating this problem
with a two-parameter family of problems of the following form
\begin{align}
&(\rhot\uvec)_t+\mbox{div}(\rhot\uvec\otimes\uvec)-2\mbox{div}\big(\nu(\varphi)D\uvec\big)+\delta A^3\uvec+\nabla\pi+\mbox{div}(\uvec\otimes\Jvect)
\nonumber\\
&+\frac{1}{2}\rhot''(\varphi)m(\varphi)(\nabla\varphi\cdot\nabla\mu)\uvec=\mu\nabla\varphi+\hvec,\label{Pb01}\\
&\mbox{div}(\uvec)=0,\label{Pb02}\\
&\varphi_t+\uvec\cdot\nabla\varphi=\mbox{div}(m(\varphi)\nabla\mu),\label{Pb03}\\
&\mu=a\varphi-J\ast\varphi+F_\e'(\varphi)-\delta\Delta\varphi,\label{Pb04}\\
&\Jvect:=-\rhot'(\varphi)m(\varphi)\nabla\mu,\label{Pb05}\\
&\uvec=0,\qquad\frac{\partial\mu}{\partial\nvec}=\frac{\partial\varphi}{\partial\nvec}=0,\qquad\mbox{on }\Gamma,\label{Pb06}\\
&\uvec(0)=\uvec_0,\qquad\varphi(0)=\varphi_{0\delta},\label{Pb07}
\end{align}
where $\e$ and $\delta$ are two fixed parameter. Notice that together with the regularizing term $-\delta\Delta\varphi$
introduced into the chemical potential, a homogeneous Neumann boundary condition for $\varphi$
has to be introduced into the approximate problem and, moreover, the initial datum for $\varphi$ has
to be suitably approximated.
The existence of a weak solution to the original problem will then be recovered
by passing to the limit in two steps in \eqref{Pb01}--\eqref{Pb07}, i.e.,
by first passing to the limit as $\e\to 0$ (with $\delta$ fixed) and then as $\delta\to 0$.
But, of course, before doing this we must prove that problem \eqref{Pb01}--\eqref{Pb07}
(for $\e$ and $\delta$ fixed) admits a weak solution. This will be achieved as first step by
means of a Faedo-Galerkin procedure.

The plan of the paper is as follows: in Section \ref{Prelim} we introduce some notation,
recall some classical results and preliminary lemmas. In Section \ref{main res} we formulate
the assumptions, the definition of weak solution and we state the main result on existence of weak solutions.
Section \ref{proof of main res} is entirely devoted to the proof of the main result. Since, as explained above,
the proof is accomplished by a three level approximation of the original system, Section \ref{proof of main res}
has been split into three subsections for each step of the approximation argument: in Subsection \ref{FG}
we develop the Faedo-Galerkin approximation scheme to prove existence of a solution to problem \eqref{Pb01}--\eqref{Pb07};
in Subsection \ref{lim-eps} we derive uniform in $\e$ estimates that allow to pass to the limit as $\e\to 0$,
in Subsection \ref{lim-delta} we obtain uniform in $\delta$ bounds, we shall pass to the limit as $\delta\to 0$
and conclude the proof.

\section{Preliminaries}\label{Prelim}

Throughout the paper, we set $H:=L^2(\Omega)$, $V:=H^1(\Omega)$, and we denote by $\Vert\,\cdot\,\Vert$
and $(\cdot\,,\,\cdot)$ the standard norm and the scalar product, respectively, in $H$
as well as in $L^2(\Omega)^d$ and $L^2(\Omega)^{d\times d}$.
The notations $\langle\cdot\,,\,\cdot\rangle_{X}$ and $\Vert\,\cdot\,\Vert_X$ will stand for the duality pairing between a Banach space
$X$ and its dual $X'$, and for the norm of $X$, respectively.

We introduce the standard Hilbert spaces for the Navier-Stokes
equations (see, e.g., \cite{T})
$$
G_{div}:=\overline{\mathcal{V}}^{L^2(\Omega)^d},\qquad
V_{div}:=\overline{\mathcal{V}}^{H^1_0(\Omega)^d},\qquad\mathcal{V}:=\{\uvec\in
C^\infty_0(\Omega)^d:\mbox{
div}(\uvec)=0\},
$$
and recall that these spaces, for Lipschitz bounded domains, can be characterized in the following way
$$
G_{div}:=\{\uvec\in
L^2(\Omega)^d:\mbox{div}(\uvec)=0,\:\:\uvec\cdot\nvec|_{\partial\Omega}=0\},\quad
V_{div}:=\{\uvec\in H_0^1(\Omega)^d:\mbox{ div}(\uvec)=0\}.
$$
The norm and scalar product in $G_{div}$ will be denoted again by
$\Vert\,\cdot\,\Vert$ and $(\cdot\,,\,\cdot)$, respectively, and
the space $V_{div}$ is endowed with the scalar product
\begin{align}
&(\uvec_1,\uvec_2)_{V_{div}}:=(\nabla\uvec_1,\nabla\uvec_2)=2\big(D\uvec_1,D\uvec_2\big),
\qquad\forall\,\uvec_1,\uvec_2\in V_{div}.\nonumber
\end{align}
We also introduce the Stokes operator $\,A\,$ with no-slip boundary condition
(see, e.g., \cite{T}). Recall that
$\,A:D(A)\subset G_{div}\to G_{div}\,$ is defined as $\,A:=-P\Delta$, with domain
$\,D(A)=H^2(\Omega)^d\cap V_{div}$,
where $\,P:L^2(\Omega)^d\to G_{div}\,$ is the Leray projector. Moreover, $\,A^{-1}:G_{div}\to G_{div}$
is a selfadjoint compact operator in $G_{div}$. Therefore, according to classical results,
$A$ possesses a sequence of eigenvalues $\{\lambda_j\}_{j\in\mathbb{N}}$ with $0<\lambda_1\leq\lambda_2\leq\cdots$ and $\lambda_j\to\infty$,
and a family $\{\wvec_j\}_{j\in \mathbb{N}}\subset D(A)$ of associated eigenfunctions which is an orthonormal
basis in $G_{div}$.
Moreover, by means of spectral theory the fractional operators $A^s$ are defined for every $s\in\mathbb{R}$
with domains $D(A^{s/2})$, which are Hilbert spaces endowed
with their natural norm and scalar product. Recall that, since $\Omega$ is assumed to be smooth, then we have
$D(A^{s/2})\hookrightarrow H^s(\Omega)^d$, for all $s\geq 0$.

We also recall Poincar\'{e}'s inequality
\begin{align}
&\lambda_1\,\Vert\uvec\Vert^2\leq\Vert\nabla\uvec\Vert^2\qquad\forall\,\uvec\in V_{div}\,.\nonumber
\end{align}


We will also need to use the operator $\,B:=-\Delta+I\,$ with homogeneous Neumann boundary condition. It is well known that $\,B:D(B)\subset H\to H\,$ is an unbounded linear operator in $\,H\,$ with the domain
$$D(B)=\big\{\varphi\in H^2(\Omega):\:\:\partial\varphi/\partial\nvec=0\,\,
\mbox{ on }\partial\Omega\big\},$$
and that $B^{-1}:H\to H$ is a selfadjoint compact operator on $H$. By a classical spectral theorem there exist a sequence of eigenvalues $\mu_j$ with $0<\mu_1\leq\mu_2\leq\cdots$ and $\mu_j\to\infty$,
and a family of associated eigenfunctions $w_j\in D(B)$ such that $Bw_j=\mu_j\, w_j\,$ for all
$j\in \mathbb{N}$. The family  $\,\{w_j\}_{j \in\mathbb{N}}\,$ forms an orthonormal basis in
$H$ and is also orthogonal in $V$ and $D(B)$.

Furthermore, for every $f\in V'$ we denote by $\overline{f}$
the average of $f$ over $\Omega$, i.e.,
$\overline{f}:=|\Omega|^{-1}\langle f,1\rangle_V$
(here $|\Omega|$ stands for the Lebesgue measure of
$\Omega$), and we introduce the spaces
$$V_0:=\{v\in V:\overline{v}=0\},\qquad V_0':=\{f\in V':\overline{f}=0\}.$$
If $m\in C(\mathbb{R})$ satisfies $m_\ast\leq m(s)\leq m^\ast$ for all $s\in\mathbb{R}$, with $m_\ast,m^\ast>0$,
then, for every measurable $\varphi:\Omega\to\mathbb{R}$ we can define the operator
$\mathcal{B}_{\varphi}:V\to V'$ by
\begin{align*}
&\langle\mathcal{B}_{\varphi} u,v\rangle_V:=\int_\Omega m(\varphi)\nabla u\cdot\nabla v,\qquad\forall u,v\in V.
\end{align*}
For every measurable $\varphi$ this operator maps $V$ onto $V_0'$ and its restriction to $V_0$
(still denoted by $\mathcal{B}_{\varphi}$)
maps $V_0$ onto $V_0'$ isomorphically. Let us denote by $\mathcal{N}_\varphi:V_0'\to V_0$ the inverse map
defined by
$$\mathcal{B}_\varphi\mathcal{N}_\varphi f=f,\quad\forall f\in V_0'\qquad
\mbox{and}\qquad\mathcal{N}_\varphi\mathcal{B}_\varphi u=u,\quad\forall u\in V_0.$$
As is well known, for every $f\in V_0'$ and every measurable $\varphi$, $\mathcal{N}_\varphi f$ is the unique
solution with zero mean value of the Neumann problem
\begin{equation*}
\left\{\begin{array}{ll}
-\mbox{div}(m(\varphi)\nabla u)=f,\qquad\mbox{in }\Omega,\\
\frac{\partial u}{\partial n}=0,\qquad\mbox{on }\partial\Omega.
\end{array}\right.
\end{equation*}
Furthermore, the following relations hold
\begin{align}
&\langle \mathcal{B}_\varphi u,\mathcal{N}_\varphi f\rangle_V=\langle f,u\rangle_V,
\qquad\forall u\in V,\quad\forall f\in V_0',\label{mathN1}\\
&\langle f,\mathcal{N}_\varphi g\rangle_V=\langle g,\mathcal{N}_\varphi
 f\rangle_V=\int_{\Omega}m(\varphi)\nabla(\mathcal{N}_\varphi f)
\cdot\nabla(\mathcal{N}_\varphi g),\qquad\forall f,g\in V_0'.\label{mathN2}
\end{align}
It is also easy to see that, for every measurable $\varphi$, we have
\begin{align}
&\frac{1}{m^\ast}\Vert f\Vert_{V'}\leq\Vert\mathcal{N}_\varphi f\Vert_V\leq\frac{1}{m_\ast}\Vert f\Vert_{V'},\qquad\forall f\in V_0'.
\label{mathN}
\end{align}

We end this section recalling threee lemmas that shall be helpful in the analysis.

The first one is a simple lemma which will be useful for passing to the limit in the variable viscosity and mobility terms
of the energy inequality. Its proof is left to the reader.
\begin{lem}\label{simplelem}
Let $Q\subset\mathbb{R}^N$, $N\geq 1$, and let $\{f_n\}\subset L^\infty(Q)$ be a sequence such that
$\Vert f_n\Vert_{L^\infty(Q)}\leq C$ and $f_n\to f$ strongly in $L^2(Q)$.
Let $\{g_n\}\subset L^2(Q)$ be another sequence such that $g_n\rightharpoonup g$ weakly in $L^2(Q)$.
Then $f_n g_n\rightharpoonup fg$ weakly in $L^2(Q)$.
\end{lem}

The next lemma will be needed to prove the weak continuity of velocities
with values in $G_{div}$. If $X$ is a Banach space, we denote by
$C_w([0,T];X)$ the topological vector space of weakly continuous
functions $f:[0,T]\to X$.
\begin{lem}
\label{Strauss}
Let $X,Y$ be two Banach spaces such that $Y\hookrightarrow X$ and $X'\hookrightarrow Y'$ densely.
Then $L^\infty(0,T;Y)\cap C([0,T];X)\hookrightarrow C_w([0,T];Y)$.
\end{lem}
The last lemma will be useful to deduce the energy inequality.
\begin{lem}
\label{Abelslemma}
Let $\mathcal{E}:[0,T)\to\mathbb{R}$, $0<T\leq\infty$, be a lower semicontinuous function
and let $\mathcal{D}:(0,T)\to\mathbb{R}$ be an integrable function.
Assume that the inequality
\begin{align*}
\mathcal{E}(0)\omega(0)+\int_0^T\mathcal{E}(\tau)\omega'(\tau)d\tau\geq\int_0^T\mathcal{D}(\tau)\omega(\tau)d\tau
\end{align*}
holds for all $\omega\in W^{1,1}(0,T)$ with $\omega(T)=0$ and $\omega\geq 0$. Then, we have
\begin{align*}
&\mathcal{E}(t)+\int_s^t\mathcal{D}(\tau)d\tau\leq\mathcal{E}(s),
\end{align*}
for almost all $s\in[0,T)$, including $s=0$, and for all $t\in[s,T)$.
\end{lem}
For a proof of the last two lemmas see, e.g., \cite{Abels2}.

Throughout the paper we shall denote by $c$, $C$,.. some nonnegative constants the value
of which may possibly change even within the same line. Generally, the value
of these constants depend on the parameters of the problem (e.g., $F$, $J$, $\nu$, $m$, $\widetilde{\rho}_i$, $\Omega$)
and on the data $\uvec_0$, $\varphi_0$, $\hvec$. Further or particular dependencies will be specified on occurrence.


\section{Main result}\label{main res}

In this section we state the main result on existence of weak solutions of system
\eqref{Pbor1}--\eqref{Pbor8}. The assumptions on the kernel $J$, on the mobility $m$ and on the viscosity $\nu$
are the following
\begin{description}
\item[(A1)]$J\in W^{1,1}(\mathbb{R}^d),\quad
    J(x)=J(-x),\quad a(x) := \displaystyle
\int_{\Omega}J(x-y)dy \geq 0,\quad\mbox{a.e. } x\in\Omega$.
\item[(A2)] $m\in C^{1,1}_{loc}(\mathbb{R})$
and there exist $m_\ast,m^\ast>0$ such that
\begin{align}
m_\ast\leq m(s)\leq m^\ast,\qquad\forall s\in\mathbb{R}.\nonumber
\end{align}
\item[(A3)] $\nu\in C^{0,1}_{loc}(\mathbb{R})$ and there exist $\nu_\ast,\nu^{\ast}>0$ such that
\begin{align*}
&\nu_\ast\leq\nu(s)\leq \nu^\ast,\qquad\forall s\in \mathbb{R}.
\end{align*}
\end{description}
As far as the singular potential $F$ is concerned we shall work
under the same assumptions as in \cite{FG2}. More precisely, we assume that
$F$ can be written in the form
$$F=F_1+F_2,$$
where $F_1\in C^{(p)}(-1,1)$, for some fixed integer $p\geq 3$, $F_2\in C^{2,1}([-1,1])$,
and that the following conditions are satisfied
\begin{description}
\item[(A4)]
There exist $c_1>0$ and $\e_0>0$ such that
\begin{align*}
&F_1^{(p)}(s)\geq c_1,\qquad\forall s\in(-1,-1+\e_0]\cup[1-\e_0,1).
\end{align*}

\item[(A5)] There exists $\e_0>0$ such that, for each
    $k=0,1,\cdots, p$ and each $j=0,1,\cdots, (p-2)/2$,
\begin{align*}
&F_1^{(k)}(s)\geq 0,\qquad\forall s\in[1-\e_0,1),\\
&F_1^{(2j+2)}(s)\geq 0,\qquad F_1^{(2j+1)}(s)\leq 0,\qquad\forall s\in(-1,-1+\e_0].
\end{align*}

\item[(A6)] There exists $\e_0>0$ such that $F_1^{(p)}$ is
    non-decreasing in $[1-\e_0,1)$
and non-increasing in $(-1,-1+\e_0]$.

\item[(A7)] There exists
$c_0>0$ such that
\begin{align*}
&F''(s)+a(x)\geq c_0,\qquad\forall s\in(-1,1),\qquad\mbox{a.a. }x\in\Omega.
\end{align*}

\item[(A8)]$\lim_{s\to\pm 1}F_1'(s)=\pm\infty.$


\end{description}

Finally, the assumption on the external force $\hvec$ is
\begin{description}
\item[(A9)]$\hvec\in L^2(0,T;V_{div}')$, for all $T>0$.
\end{description}

\begin{oss}
{\upshape
Assumption $J\in W^{1,1}(\mathbb{R}^d)$ can be weakened. Indeed, the behavior of the kernel at infinity is not essential.
Alternative conditions are $J\in W^{1,1}(B_{\delta})$, where $B_{\delta}:=\{z\in\mathbb{R}%
^d:|z|<\delta\}$ with $\delta:=\mbox{diam}(\Omega)$, or $J\in W^{1,1}(\Omega-\Omega)$, where
$\Omega-\Omega:=\{z\in\mathbb{R}^d: z=x-y,\:\: x,y\in\Omega\}$ or also (see, e.g.,
\cite{BH1})
\begin{equation*}
\sup_{x\in\Omega}\int_{\Omega}\big(|J(x-y)|+|\nabla J(x-y)|\big)dy<\infty.
\end{equation*}
}
\end{oss}

\begin{oss}
{\upshape
Assumptions (A4)-(A8) are satisfied in the case of the
physically relevant logarithmic double-well potential \eqref{log} for
every fixed integer $p\geq 3$. In particular, setting
$$F_1(s)=\frac{\theta}{2}((1+s)\log(1+s)+(1-s)\log(1-s)),\qquad F_2(s)=-\frac{\theta_c}{2}s^2,$$
then it is easy to check that (A7) is satisfied if and only if $\inf_\Omega a>\theta_c-\theta$. However, note that other reasonable potentials satisfy the above assumptions (e.g., the ones which are unbounded at the endpoints).
}
\end{oss}

Let us state now the notion of weak solution to Problem \eqref{Pbor1}--\eqref{Pbor8}.

\begin{defn}
\label{wsdefn}
Let $\uvec_0\in G_{div}$, $\varphi_0\in L^\infty(\Omega)$ with $F(\varphi_0)\in L^1(\Omega)$ and $0<T<+\infty$ be given.
A couple $[\uvec,\varphi]$ is a weak solution to \eqref{Pbor1}-\eqref{Pbor8} on $[0,T]$ corresponding to $[\uvec_0,\varphi_0]$
if
\begin{itemize}
\item  $\uvec$, $\varphi$ and $\mu$ satisfy
\begin{align}
&\uvec\in C_w([0,T];G_{div})\cap L^2(0,T;V_{div}),\label{rewsol1}\\
&\varphi\in L^{\infty}(0,T;H)\cap L^2(0,T;V),\label{rewsol2}\\
&\mu=a\varphi-J\ast\varphi+F'(\varphi)\in L^2(0,T;V),\label{rewsol3}\\
&(\rho\uvec)_t\in L^{4/3}(0,T;D(A)'),\qquad\varphi_t\in L^2(0,T;V'),\label{rewsol4}
\end{align}
and
\begin{align}
&\varphi\in L^{\infty}(Q),\qquad|\varphi(x,t)|<1\quad\mbox{a.e. }(x,t)\in Q:=\Omega\times(0,T);\label{rewsol5}
\end{align}

\item for every $\psi\in V$, every $\wvec\in D(A)$ and for almost
any $t\in(0,T)$ we have
\begin{align}
&\big\langle (\rho\uvec)_t,\wvec\big\rangle_{D(A)}-(\rho\uvec\otimes\uvec,D\wvec)+(2\nu(\varphi)D\uvec,D\wvec)
-(\uvec\otimes\Jvect,D\wvec)\nonumber\\
&=-(\varphi\nabla\mu,\wvec)+\langle \hvec,\wvec\rangle_{V_{div}},\label{weakfor1}\\
&\langle\varphi_t,\psi\rangle_V+(m(\varphi)\nabla\mu,\nabla\psi)=(\uvec,\varphi\nabla\psi),\label{weakfor2}
\end{align}
where $\Jvect=-\beta m(\varphi)\nabla\mu \in L^2(0,T;H)$;

\item the initial conditions $\uvec(0)=\uvec_0$, $\varphi(0)=\varphi_0$ hold.

\end{itemize}

\end{defn}
\begin{oss}
{\upshape
Notice that \eqref{rewsol2} and the second of \eqref{rewsol4} imply that $\varphi\in C([0,T];H)$.
Hence, thanks also to \eqref{rewsol1}, the initial conditions $\uvec(0)=\uvec_0$, $\varphi(0)=\varphi_0$
make sense.
}
\end{oss}

We are now ready to state the main result.

\begin{thm}
\label{mainres}
Assume that (A1)--(A9) are satisfied for some fixed integer $p\geq 3$, and $d=2,3$. Let $\uvec_0\in G_{div}$, $\varphi_0\in L^\infty(\Omega)$ such that
$F(\varphi_0)\in L^1(\Omega)$ and $|\overline{\varphi}_0|<1$. Then, for every $T>0$ Problem \eqref{Pbor1}--\eqref{Pbor8}
admits a weak solution $[\uvec,\varphi]$ on $[0,T]$ corresponding to $\uvec_0$, $\varphi_0$
such that
\begin{align}
&\varphi\in L^\infty(0,T;L^p(\Omega)),\label{rewsol6}
\end{align}
and satisfying the following energy inequality
\begin{align}
&\mathcal{E}(\uvec(t),\varphi(t))
+\int_s^t \big(2\Vert\sqrt{\nu(\varphi)}D\uvec\Vert^2 d\tau +\Vert\sqrt{m(\varphi)}\nabla\mu\Vert^2\big) d\tau
\leq\mathcal{E}(\uvec(s),\varphi(s))+\int_s^t\langle \hvec,\uvec\rangle_{V_{div}} d\tau,
\label{eninPbor}
\end{align}
for almost all $s\in[0,T)$, including $s=0$, and for all $t\in[s,T]$, where
\begin{align*}
&\mathcal{E}(\uvec,\varphi):=\int_\Omega\frac{1}{2}\rho(\varphi)\uvec^2+E(\varphi),\\
&E(\varphi):=\frac{1}{2}\Vert\sqrt{a}\varphi\Vert^2-\frac{1}{2}(\varphi,J\ast\varphi)+\int_\Omega F(\varphi)
=\frac{1}{4}\int_\Omega\int_\Omega J(x-y)\big(\varphi(x)-\varphi(y)\big)^2 dxdy+\int_\Omega F(\varphi).
\end{align*}
\end{thm}


\section{Proof of the main result}\label{proof of main res}

The proof will be carried out in three steps. In the first step we shall consider the two parameters
approximate problem P$_{\epsilon,\delta}$ given by \eqref{Pb01}--\eqref{Pb07} (with both $\epsilon,\delta>0$ fixed)
and shall implement a Faedo-Galerkin
approximation scheme to prove existence of a global weak solution $[\uvec_{\e,\delta},\varphi_{\e,\delta}]$ to P$_{\epsilon,\delta}$ satisfying an energy inequality.
In the second step we shall consider only $\delta>0$ fixed and deduce some uniform in $\e$ bounds for the weak solution
(that we can now denote by $[\uvec_\e,\varphi_\e]$)
to problem P$_{\epsilon,\delta}$ which will allow to pass to the limit as $\e\to 0$
in the weak formulation of P$_{\e,\delta}$ and to prove that the family of solutions $[\uvec_\e,\varphi_\e]$
converges to a solution (that now we can denote by $[\uvec_\delta,\varphi_\delta]$) to Problem P$_\delta$
given by
\begin{align}
&(\rho\uvec)_t+\mbox{div}(\rho\uvec\otimes\uvec)-2\mbox{div}\big(\nu(\varphi)D\uvec\big)+\delta A^3\uvec+\nabla\pi+\mbox{div}(\uvec\otimes\Jvect)
=\mu\nabla\varphi+\hvec,\label{Pb11}\\
&\mbox{div}(\uvec)=0,\label{Pb12}\\
&\varphi_t+\uvec\cdot\nabla\varphi=\mbox{div}(m(\varphi)\nabla\mu),\label{Pb13}\\
&\mu=a\varphi-J\ast\varphi+F'(\varphi)-\delta\Delta\varphi,\label{Pb14}\\
&\Jvect:=-\beta m(\varphi)\nabla\mu,\label{Pb15}\\
&\uvec=0,\qquad\frac{\partial\mu}{\partial\nvec}=\frac{\partial\varphi}{\partial\nvec}=0,\qquad\mbox{on }\partial\Omega,\label{Pb16}\\
&\uvec(0)=\uvec_0,\qquad\varphi(0)=\varphi_{0\delta},\label{Pb17}
\end{align}
in which the potential $F$ is singular.
Finally, in the third step we shall deduce uniform in $\delta$ estimates for
the family of solutions $[\uvec_\delta,\varphi_\delta]$
to Problem P$_\delta$ and pass to the limit as $\delta\to 0$ to prove that
$[\uvec_\delta,\varphi_\delta]$ converges to a solution to the original problem \eqref{Pbor1}--\eqref{Pbor8}.

In all the analysis we shall consider only the case $d=3$. If $d=2$ all the steps
of the proof of Theorem \ref{mainres} can be repeated (with strong convergences in stronger norms
in comparison with the 3D case). However, the result of Theorem \ref{mainres}
does not improve substantially in 2D (see Remark \ref{2Dcase}).

\subsection{Step I. Faedo-Galerkin approximation scheme.}\label{FG}

For problem \eqref{Pb01}--\eqref{Pb07}
we shall consider the general situation of a regular potential $F_\e$, that in this subsection we denote simply
by $F$, of arbitrary polynomial growth. Therefore, the assumptions
we make for $F$ are the following (cf. \cite{CFG})
\begin{description}
 \item[(RP1)] $F\in C^{2,1}_{loc}(\mathbb{R})$ and there exists $c_0>0$
     such that
             $$F^{\prime\prime}(s)+a(x)\geq c_0,\qquad\forall s\in\mathbb{R},\quad\mbox{a.e. }x\in\Omega.$$
 \item[(RP2)] $F\in C^2(\mathbb{R})$ and there exist $\hat{c}_1>0$,
    $\hat{c}_2>0$ and $p\geq 3$ such that
            $$F^{\prime\prime}(s)+a(x)\geq \hat{c}_1\vert s\vert^{p-2} - \hat{c}_2,
            \qquad\forall s\in\mathbb{R},\quad\mbox{a.e. }x\in\Omega.$$
 \item[(RP3)] There exist $\hat{c}_3>0$, $\hat{c}_4\geq0$ and $r\in(1,2]$
     such that
             $$|F^\prime(s)|^r\leq \hat{c}_3|F(s)|+\hat{c}_4,\qquad
             \forall s\in\mathbb{R}.$$
\end{description}

\begin{oss}
{\upshape
Since $F$ is bounded from below, it is easy to see that (RP3) implies that $F$
has polynomial growth of order $r^{\prime}$, where $%
r^{\prime}\in[2,\infty)$ is the conjugate index to $r$.
Namely, there exist $\hat{c}_5>0$ and $\hat{c}_6\geq 0$ such that
\begin{equation}  \label{growth}
|F(s)|\leq \hat{c}_5|s|^{r^{\prime}}+\hat{c}_6,\qquad\forall s\in\mathbb{R}.
\end{equation}
Observe that assumption (RP3) is fulfilled by a potential of arbitrary
polynomial growth.
}

\end{oss}

The assumptions on the kernel $J$ and on the mobility $m$ and on the viscosity $\nu$ are
the same as (A1), (A2) and (A3), respectively.

A weak solution to Problem P$_{\e,\delta}$ is a pair $[\uvec,\varphi]$ satisfying \eqref{Pb01}--\eqref{Pb07}
in the following sense
\begin{defn}
\label{wsdefn2}
Let $\uvec_0\in G_{div}$, $\varphi_0\in L^\infty(\Omega)$ with $F(\varphi_0)\in L^1(\Omega)$ and $0<T<+\infty$ be given.
A pair $[\uvec,\varphi]$ is a weak solution to \eqref{Pb01}-\eqref{Pb07} on $[0,T]$ corresponding to $[\uvec_0,\varphi_0]$
if
\begin{itemize}
\item  $\uvec$, $\varphi$ and $\mu$ satisfy
\begin{align}
&\uvec\in C_w([0,T];G_{div})\cap L^2(0,T;D(A^{3/2})),\label{regp1Pb0}\\
&\varphi\in L^\infty(0,T;V)\cap L^2(0,T;H^2(\Omega)),\label{regp2Pb0}\\
&\mu=a\varphi-J\ast\varphi+F'(\varphi)-\delta\Delta\varphi\in L^2(0,T;V),\label{regp3Pb0}\\
&(\rhot\uvec)_t\in L^\kappa(0,T;D(A^{3/2})'),\qquad\varphi_t\in L^2(0,T;V'),\label{regp4Pb0}
\end{align}
for some $\kappa>1$;
\item for every $\psi\in V$, every $\wvec\in D(A^{3/2})$ and for almost
any $t\in(0,T)$ we have
\begin{align}
&\big\langle (\rhot\uvec)_t,\wvec\big\rangle_{D(A^{3/2})}-(\rhot\uvec\otimes\uvec,D\wvec)+(2\nu(\varphi)D\uvec,D\wvec)
+\delta(A^{3/2}\uvec,A^{3/2}\wvec)\nonumber\\
&-(\uvec\otimes\Jvect,D\wvec)+\frac{1}{2}\big(\rhot''(\varphi)m(\varphi)(\nabla\varphi\cdot\nabla\mu)\uvec,\wvec\big)
=-(\varphi\nabla\mu,\wvec)+\langle \hvec,\wvec\rangle_{V_{div}},\label{wf1step1}\\
&\langle\varphi_t,\psi\rangle_V+(m(\varphi)\nabla\mu,\nabla\psi)=(\uvec,\varphi\nabla\psi),\label{wf2step1}
\end{align}
where $\Jvect=-\rhot'(\varphi) m(\varphi)\nabla\mu \in L^2(0,T;H)$;

\item the initial conditions $\uvec(0)=\uvec_0$, $\varphi(0)=\varphi_0$ hold.

\end{itemize}

\end{defn}
We are now ready to state the following
\begin{lem}
\label{existence1}
Let assumptions (A1)--(A3), (A9) and (RP1)--(RP3) be satisfied.
Let $\uvec_0\in G_{div}$ and $\varphi_0\in V$ such that $F(\varphi_0)\in L^1(\Omega)$.
Then, for every $T>0$ Problem \eqref{Pb01}--\eqref{Pb07}
admits a weak solution $[\uvec,\varphi]$ on $[0,T]$
corresponding to $[\uvec_0,\varphi_0]$
such that
\begin{align}
&\varphi\in L^\infty(0,T;L^p(\Omega)),\label{regp2Pbobis}
\end{align}
satisfying the following energy inequality
\begin{align}
&\int_\Omega\frac{1}{2}\rhot(\varphi(t))\uvec^2(t)+E(\varphi(t))+\frac{\delta}{2}\Vert\nabla\varphi(t)\Vert^2+
2\int_0^t\Vert\sqrt{\nu(\varphi)}D\uvec\Vert^2d\tau
+\delta\int_0^t\Vert A^{3/2}\uvec\Vert^2d\tau
\nonumber\\
&+\int_0^t\big\Vert\sqrt{m(\varphi)}\nabla\mu\big\Vert^2d\tau
\leq\int_\Omega\frac{1}{2}\rhot(\varphi_0)\uvec_0^2+E(\varphi_0)
+\frac{\delta}{2}\Vert\nabla\varphi_0\Vert^2+\int_0^t\langle \hvec,\uvec\rangle_{V_{div}} d\tau,
\label{enin}
\end{align}
for almost all $t\in (0,T)$,
where we have set
\begin{align*}
&E(\varphi):=\frac{1}{2}\Vert\sqrt{a}\varphi\Vert^2-\frac{1}{2}(\varphi,J\ast\varphi)+\int_\Omega F(\varphi)
=\frac{1}{4}\int_\Omega\int_\Omega J(x-y)\big(\varphi(x)-\varphi(y)\big)^2 dxdy+\int_\Omega F(\varphi).
\end{align*}

\end{lem}
\begin{proof}
Let us assume in addition that $\varphi_0\in D(B)$.
Existence of a weak solution as well as the energy inequality
in the more general case of $\varphi_0\in V$ with $F(\varphi_0)\in L^1(\Omega)$
can be recovered by means of a density argument, in the same fashion as in \cite[Proof of Theorem 1]{CFG},
by exploiting in particular the fact that, due to (RP1), $F$ is a quadratic perturbation of a convex function.
We introduce the family $\{\wvec_j\}_{j\geq 1}$ of the eigenfunctions of the Stokes operator $A$ as a Galerkin
base in $V_{div}$ and the family $\{\psi_j\}_{j\geq 1}$ of the eigenfunctions of $B$
as a Galerkin base in $V$. We define the $n-$dimensional subspaces
$\mathcal{W}_n:=\langle \wvec_1,\cdots,\wvec_n\rangle$ and
$\Psi_n:=\langle\psi_1,\cdots,\psi_n\rangle$ and consider the
orthogonal projectors on these subspaces in $G_{div}$ and $H$,
respectively, i.e., $\widetilde{P}_n:=P_{\mathcal{W}_n}$ and
$P_n:=P_{\Psi_n}$.

We then look for three functions of the form
\begin{align}
&\uvec_n(t)=\sum_{j=i}^n a_j^{(n)}(t)\wvec_j,\qquad \varphi_n(t)=\sum_{j=1}^n b_j^{(n)}(t)\psi_j,
\qquad \mu_n(t)=\sum_{j=1}^n c_j^{(n)}(t)\psi_j,
\notag
\end{align}
that solve the following approximating problem
\begin{align}
&\Big(\big(\rhot(\varphi_n)\uvec_n\big)',\wvec_k\Big)-\big(\rhot(\varphi_n)\uvec_n\otimes \uvec_n,D\wvec_k\big)
+2\big(\nu(\varphi_n)D\uvec_n,D\wvec_k\big)+\delta(A^{3/2}\uvec_n,A^{3/2}\wvec_k)\nonumber\\
&-\int_\Omega \uvec_n\cdot (\Jvect_n\cdot\nabla)\wvec_k
+\frac{1}{2}\big(\rhot''(\varphi_n) m(\varphi_n)(\nabla\varphi_n\cdot\nabla\mu_n)\uvec_n,\wvec_k\big)\nonumber\\
&+\frac{1}{2}\Big(\rhot'(\varphi_n)\big(P_n(\uvec_n\cdot\nabla\varphi_n)-\uvec_n\cdot\nabla\varphi_n\big)\uvec_n,\wvec_k\Big)\nonumber\\
&+\frac{1}{2}\Big(\rhot'(\varphi_n)\big(\mbox{div}(m(\varphi_n)\nabla\mu_n)-P_n(\mbox{div}(m(\varphi_n)\nabla\mu_n))\big)\uvec_n,\wvec_k\Big)\nonumber\\
&=-(\varphi_n\nabla\mu_n,\wvec_k)+\langle\hvec_n,\wvec_k\rangle_{V_{div}},\quad k=1,\cdots,n
\label{FGapp1}\\
&(\varphi_n',\psi_k)+(m(\varphi_n)\nabla\mu_n,\nabla\psi_k)=(\uvec_n\varphi_n,\nabla\psi_k),\qquad k=1,\cdots,n
\label{FGapp2}\\
&\mu_n=P_n(a\varphi_n-J\ast\varphi_n+F'(\varphi_n)-\delta\Delta\varphi_n)
\label{FGapp3}\\
&\Jvect_n:=-\rhot'(\varphi_n)m(\varphi_n)\nabla\mu_n
\label{FGapp4}\\
&\varphi_n(0)=\varphi_{0n},\qquad \uvec_n(0)=\uvec_{0n},
\label{FGapp5}
\end{align}
where primes denote derivatives with respect to time, $\varphi_{0n}= P_n\varphi_0$, $\uvec_n(0)=\widetilde{P}_n\uvec_{0}$,
and $\hvec_n\in C([0,T];G_{div})$ such that $\hvec_n\to\hvec$ in $L^2(0,T;V_{div}')$.
Moreover, we assume that the function $\rhot$ satisfying \eqref{bdrho1} is fixed
 such that $\rhot\in C^{2,1}_{loc}(\mathbb{R})$.
By writing
\begin{align*}
&\big(\rhot(\varphi_n)\uvec_n\big)'=\rhot(\varphi_n)\uvec_n'+\rhot'(\varphi_n)\uvec_n\varphi_n'
\end{align*}
and observing that \eqref{FGapp2} can be written as
\begin{align*}
&\varphi_n'=P_n\big(\mbox{div}(m(\varphi_n)\nabla\mu_n)-\uvec_n\cdot\nabla\varphi_n\big),
\end{align*}
it is not difficult to see that solving this approximating problem is equivalent
to solving a system of ordinary differential equations in the $2n$ unknowns $a_j^{(n)}$, $b_j^{(n)}$
which can be reduced in normal form thanks to the fact that we have $\rhot(s)\geq\rho_\ast$, for all $s\in\mathbb{R}$,
with $\rho_\ast>0$. Indeed, this condition ensures that, for every fixed $n\in\mathbb{N}$,
the vectors $\sqrt{\rhot(\varphi_n)}\wvec_1,\cdots,\sqrt{\rhot(\varphi_n)}\wvec_n$ are linearly
independent and hence the Gram matrix $\{\big(\rhot(\varphi_n)\wvec_j,\wvec_k\big)\}_{j,k=1,..n}$,
which appears in the first term on the left hand side of \eqref{FGapp1} when it is written
explicitly in terms of the unknowns $a_j^{(n)}$,
is not singular.

By the Cauchy-Lipschitz theorem we know that there exists $T_n^\ast\in(0,+\infty]$ such that this system admits a unique maximal solution
$\boldsymbol{a}^{(n)}:=(a_1^{(n)},\cdots,a_n^{(n)})$, $\boldsymbol{b}^{(n)}:=(b_1^{(n)},\cdots,b_n^{(n)})$
on $[0,T_n^\ast)$ with $\boldsymbol{a}^{(n)},\boldsymbol{b}^{(n)}\in C^1([0,T_n^\ast);\mathbb{R}^n)$.

We now multiply \eqref{FGapp1} by $a_k^{(n)}$, \eqref{FGapp2} by $c_k^{(n)}$ and sum over $k=1,\cdots,n$ taking
\eqref{FGapp3} and \eqref{FGapp4} into account. In doing this we also observe that the following
identity holds
\begin{align}
&\Big(\big(\rhot(\varphi_n)\uvec_n\big)',\uvec_n\Big)
=\frac{d}{dt}\int_\Omega\frac{1}{2}\rhot(\varphi_n)\uvec_n^2+\frac{1}{2}\int_\Omega\rhot(\varphi_n)_t\uvec_n^2\nonumber\\
&=\frac{d}{dt}\int_\Omega\frac{1}{2}\rhot(\varphi_n)\uvec_n^2
+\frac{1}{2}\int_\Omega\rhot^\prime(\varphi_n)P_n\big(\mbox{div}(m(\varphi_n)\nabla\mu_n)-\uvec_n\cdot\nabla\varphi_n\big)\uvec_n^2.
\label{auxid1}
\end{align}
Moreover, on account of the incompressibility condition $\mbox{div}(\uvec_n)=0$ and
of the no-slip boundary condition $\uvec_n=0$ on $\partial\Omega$, we have
\begin{align}
&-\big(\rhot(\varphi_n)\uvec_n\otimes \uvec_n,D\uvec_n\big)=\int_\Omega \uvec_n\cdot(\uvec_n\cdot\nabla)(\rhot(\varphi_n)\uvec_n)
=\int_\Omega \uvec_n^2(\uvec_n\cdot\nabla\rhot(\varphi_n))\nonumber\\
&+\int_\Omega\rhot(\varphi_n)\uvec_n\cdot(\uvec_n\cdot\nabla)\uvec_n\nonumber
\end{align}
and the last term on the right hand side of this identity can be written as
\begin{align}
&\int_\Omega\rhot(\varphi_n)\uvec_n\cdot(\uvec_n\cdot\nabla)\uvec_n=\int_\Omega\frac{\uvec_n^2}{2}\rhot(\varphi_n)(\uvec_n\cdot\nabla)
=-\int_\Omega\frac{\uvec_n^2}{2}(\uvec_n\cdot\nabla\rhot(\varphi_n)).\nonumber
\end{align}
Therefore, we have
\begin{align}
&-\big(\rhot(\varphi_n)\uvec_n\otimes \uvec_n,D\uvec_n\big)=\int_\Omega(\uvec_n\cdot\nabla\rhot(\varphi_n))\frac{\uvec_n^2}{2}
=\frac{1}{2}\int_\Omega\rhot'(\varphi_n)(\uvec_n\cdot\nabla\varphi_n)\uvec_n^2.
\label{auxid2}
\end{align}
Furthermore we have
\begin{align}
&-\int_\Omega \uvec_n\cdot (\Jvect_n\cdot\nabla)\uvec_n=
\int_\Omega \uvec_n\cdot\mbox{div}(\uvec_n\otimes\Jvect)
=\int_\Omega\uvec_n\cdot\big[(\mbox{div}\Jvect_n)\uvec_n+(\Jvect_n\cdot\nabla)\uvec_n\big]\nonumber\\
&
=\int_\Omega\uvec_n^2\,\mbox{div}\Jvect_n+\int_\Omega\Jvect_n\cdot\nabla\Big(\frac{\uvec_n^2}{2}\Big)=
\int_\Omega\frac{1}{2}\uvec_n^2\,\mbox{div}\Jvect_n\nonumber\\
&=-\frac{1}{2}\int_\Omega\mbox{div}(\rhot'(\varphi_n)m(\varphi_n)\nabla\mu_n)\uvec_n^2\nonumber\\
&=-\frac{1}{2}\int_\Omega\rhot''(\varphi_n)m(\varphi_n)(\nabla\varphi_n\cdot\nabla\mu_n)\,\uvec_n^2
-\frac{1}{2}\int_\Omega\rhot'(\varphi_n)\mbox{div}(m(\varphi_n)\nabla\mu_n)\uvec_n^2
\label{auxid3}
\end{align}
By means of \eqref{auxid1}--\eqref{auxid3} after some easy computations we then arrive at the following identity
\begin{align}
&\frac{d}{dt}\Big(\int_\Omega\frac{1}{2}\rhot(\varphi_n)\uvec_n^2+E(\varphi_n)+\frac{\delta}{2}\Vert\nabla\varphi_n\Vert^2\Big)+
2\Vert\sqrt{\nu(\varphi_n)}D\uvec_n\Vert^2+\delta\Vert A^{3/2}\uvec_n\Vert^2\nonumber\\
&+\Vert\sqrt{m(\varphi_n)}\nabla\mu_n\Vert^2=\langle\hvec_n,\uvec_n\rangle_{V_{div}}.
\label{FGenid}
\end{align}
We can now integrate \eqref{FGenid} between 0 and $t$ (after splitting the term on the right hand side on
account of (A9))
and use (A2), (A3), \eqref{bdrho1}, (RP2) and the fact that, since $\varphi_0\in D(B)$, then we have $\varphi_{0n}\to\varphi_0$
 in $L^\infty(\Omega)$, to deduce the following bounds
\begin{align}
&\Vert\uvec_n\Vert_{L^\infty(0,T;G_{div})\cap L^2(0,T;V_{div})}\leq C,\label{bd1}\\
\delta^{1/2}&\Vert\uvec_n\Vert_{L^2(0,T;D(A^{3/2}))}\leq C,\label{bd2}\\
&\Vert\varphi_n\Vert_{L^\infty(0,T;L^p(\Omega))}\leq C,\label{bd3}\\
\delta^{1/2}&\Vert\varphi_n\Vert_{L^\infty(0,T;V)}\leq C,\label{bd4}\\
&\Vert\nabla\mu_n\Vert_{L^2(0,T;H)}\leq C ,\label{bd5}\\
&\Vert F(\varphi_n)\Vert_{L^\infty(0,T;L^1(\Omega))}\leq C,\label{bd6}
\end{align}
where henceforth in this proof we shall denote by $C$ a positive constant such that
\begin{align*}
& C=C\big(\Vert\uvec_0\Vert,\Vert\varphi_0\Vert_V,\Vert F(\varphi_0)\Vert_{L^1(\Omega)},\Vert\hvec\Vert_{L^2(0,T;V_{div}')}\big).
\end{align*}
The constant $C$ also depends on $F$, $J$, $\nu_\ast$, $m_\ast$, $\rho_\ast$ and $\Omega$.
Notice that these bounds hold at first instance with $T=T_n^\ast$. However, since we have
$\Vert\uvec_n(t)\Vert=|\boldsymbol{a}^{(n)}(t)|$ and $\Vert\varphi_n(t)\Vert=|\boldsymbol{b}^{(n)}(t)|$,
then the same bounds yield $T_n^\ast=+\infty$, i.e. $T>0$ can be fixed arbitrary and
the above bounds hold for every $T>0$.

Moreover, multiplying \eqref{FGapp3} by $-\Delta\varphi_n$ in $H$ and observing that $-\Delta\varphi_n$
belongs to the subspace $\Psi_n$, we have
\begin{align*}
&(\mu_n,-\Delta\varphi_n)=(\nabla\mu_n,\nabla\varphi_n)=(a\varphi_n-J\ast\varphi_n+F'(\varphi_n)-\delta\Delta\varphi_n, -\Delta\varphi_n)\nonumber\\
&=\big(\nabla\varphi_n,(a+F''(\varphi_n))\nabla\varphi_n+\varphi_n\nabla a-\nabla J\ast\varphi_n\big)+\delta\Vert\Delta\varphi_n\Vert^2\nonumber\\
&\geq c_0\Vert\nabla\varphi_n\Vert^2-2\Vert\nabla J\Vert_{L^1}\Vert\nabla\varphi_n\Vert\Vert\varphi_n\Vert+\delta\Vert\Delta\varphi_n\Vert^2\nonumber\\
&\geq\frac{c_0}{2}\Vert\nabla\varphi_n\Vert^2-c\Vert\varphi_n\Vert^2+\delta\Vert\Delta\varphi_n\Vert^2,
\end{align*}
and therefore, combining this estimate with
\begin{align*}
&(\nabla\mu_n,\nabla\varphi_n)\leq\frac{c_0}{4}\Vert\nabla\varphi_n\Vert^2+\frac{1}{c_0}\Vert\nabla\mu_n\Vert^2,
\end{align*}
we get
\begin{align}
&\Vert\nabla\mu_n\Vert^2\geq\frac{c_0^2}{4}\Vert\nabla\varphi_n\Vert^2-c\Vert\varphi_n\Vert^2+\delta c_0\Vert\Delta\varphi_n\Vert^2.
\label{est1}
\end{align}
By employing \eqref{est1}, and using a classical elliptic regularity result, from \eqref{bd5} and \eqref{bd3}
 we hence deduce the following estimates
\begin{align}
&\Vert\varphi_n\Vert_{L^2(0,T;V)}\leq C,\label{bd7}\\
\delta^{1/2}&\Vert\varphi_n\Vert_{L^2(0,T;H^2(\Omega))}\leq C.\label{bd8}
\end{align}
As far as the control of the sequence of the averages $\{\overline{\mu}_n\}$ is concerned, we have
\begin{align*}
&\Big|\int_\Omega\mu_n\Big|=|(a\varphi_n-J\ast\varphi_n+F'(\varphi_n)-\delta\Delta\varphi_n,1)|=|(F'(\varphi_n),1)|\nonumber\\
&\leq c\Vert F(\varphi_n)\Vert_{L^1(\Omega)}+c\leq C,
\end{align*}
due to the bound \eqref{bd6}. By means of Poincar\'{e}-Wirtinger inequality, this control and \eqref{bd5} imply the bound
\begin{align}
\Vert\mu_n\Vert_{L^2(0,T;V)}\leq C.\label{bd9}
\end{align}
Therefore, since $\delta>0$ is fixed, from the estimates obtained above
we deduce that there exist $\uvec$, $\varphi$ and $\mu$ such that, up to a subsequence we have
\begin{align}
&\uvec_n\rightharpoonup \uvec,\qquad\mbox{weakly$^\ast$ in } L^\infty(0,T;G_{div}),\label{weak1}\\
&\uvec_n\rightharpoonup \uvec,\qquad\mbox{weakly in } L^2(0,T;D(A^{3/2})),\label{weak2}\\
&\varphi_n\rightharpoonup \varphi,\qquad\mbox{weakly$^\ast$ in } L^\infty(0,T;V\cap L^p(\Omega)),\label{weak3}\\
&\varphi_n\rightharpoonup \varphi,\qquad\mbox{weakly in } L^2(0,T;H^2(\Omega)),\label{weak4}\\
&\mu_n\rightharpoonup \mu,\qquad\mbox{weakly in } L^2(0,T;V).\label{weak5}
\end{align}

As next step we need to derive some estimates for the two sequences of time derivatives
$\big\{\partial_t \widetilde{P}_n\big(\rhot(\varphi_n)\uvec_n\big)\big\}$ and $\{\varphi_n'\}$.
Let us begin with the first sequence.
Take $\wvec\in D(A^{3/2})$ and write $\wvec=\wvec_I+\wvec_{II}$, where $\wvec_I\in \mathcal{W}_n$
and $\wvec_{II}\in\mathcal{W}_n^{\bot}$
(recall that $\wvec_I$ and $\wvec_{II}$ are orthogonal in all Hilbert spaces $D(A^{s/2})$ for all $0\leq s\leq 3$).
From \eqref{FGapp1} we can write
\begin{align}
&\Big\langle\partial_t \widetilde{P}_n\big(\rhot(\varphi_n)\uvec_n\big),\wvec\Big\rangle_{D(A^{3/2})}
=\Big(\big(\rhot(\varphi_n)\uvec_n\big)',\wvec_I\Big)\nonumber\\
&=
\big(\rhot(\varphi_n)\uvec_n\otimes \uvec_n,D\wvec_I\big)
-2\big(\nu(\varphi_n)D\uvec_n,D\wvec_I\big)-\delta(A^{3/2}\uvec_n,A^{3/2}\wvec_I)\nonumber\\
&+\int_\Omega \uvec_n\cdot (\Jvect_n\cdot\nabla)\wvec_I
-\frac{1}{2}\big(\rhot''(\varphi_n) m(\varphi_n)(\nabla\varphi_n\cdot\nabla\mu_n)\uvec_n,\wvec_I\big)\nonumber\\
&-\frac{1}{2}\Big(\rhot'(\varphi_n)\big(P_n(\uvec_n\cdot\nabla\varphi_n)-\uvec_n\cdot\nabla\varphi_n\big)\uvec_n,\wvec_I\Big)\nonumber\\
&-\frac{1}{2}\Big(\rhot'(\varphi_n)\big(\mbox{div}(m(\varphi_n)\nabla\mu_n)-P_n(\mbox{div}(m(\varphi_n)\nabla\mu_n))\big)\uvec_n,\wvec_I\Big)\nonumber\\
&-(\varphi_n\nabla\mu_n,\wvec_I)+\langle\hvec_n,\wvec_I\rangle_{V_{div}}.\label{esttimeder}
\end{align}
We now estimate individually the terms on the right hand side of \eqref{esttimeder}. We have
\begin{align}
&|\big(\rhot(\varphi_n)\uvec_n\otimes \uvec_n,D\wvec_I\big)|\leq c\Vert\uvec_n\Vert^2\Vert\wvec_I\Vert_{H^3(\Omega)^3}
\leq c\Vert\uvec_n\Vert^2\Vert\wvec\Vert_{D(A^{3/2})},\label{rhoutd1}\\
&2\big|\big(\nu(\varphi_n)D\uvec_n,D\wvec_I\big)\big|\leq c\Vert\nabla\uvec_n\Vert\Vert\wvec_I\Vert_{H^1(\Omega)^3}\leq
c\Vert\nabla\uvec_n\Vert\Vert\wvec\Vert_{D(A^{1/2})},\\
&\delta|(A^{3/2}\uvec_n,A^{3/2}\wvec_I)|\leq\delta\Vert A^{3/2}\uvec_n\Vert\Vert\wvec_I\Vert_{D(A^{3/2})}
\leq\delta\Vert A^{3/2}\uvec_n\Vert\Vert\wvec\Vert_{D(A^{3/2})},\\
&\Big|\int_\Omega \uvec_n\cdot (\Jvect_n\cdot\nabla)\wvec_I\Big|
=\Big|\int_\Omega \uvec_n\cdot\big(\rhot'(\varphi_n)m(\varphi_n)\nabla\mu_n\cdot\nabla)\wvec_I\Big|\nonumber\\
&\leq c\Vert\uvec_n\Vert\Vert\nabla\mu_n\Vert\Vert\wvec_I\Vert_{H^3(\Omega)^3}\leq
c\Vert\uvec_n\Vert\Vert\nabla\mu_n\Vert\Vert\wvec\Vert_{D(A^{3/2})},\\
&\Big|\frac{1}{2}\big(\rhot''(\varphi_n) m(\varphi_n)(\nabla\varphi_n\cdot\nabla\mu_n)\uvec_n,\wvec_I\big)\Big|
\leq c\Vert\nabla\varphi_n\Vert_{L^{10/3}(\Omega)^3}\Vert\nabla\mu_n\Vert\Vert\uvec_n\Vert_{L^6(\Omega)^3}\Vert\wvec\Vert_{D(A)},\\
&\Big|\frac{1}{2}\Big(\rhot'(\varphi_n)\big(P_n(\uvec_n\cdot\nabla\varphi_n)-\uvec_n\cdot\nabla\varphi_n\big)\uvec_n,\wvec_I\Big)\Big|
\leq c\Vert\uvec_n\cdot\nabla\varphi_n\Vert\Vert\uvec_n\Vert\Vert\wvec_I\Vert_{H^2(\Omega)^3}\nonumber\\
&\leq c\Vert\uvec_n\Vert_{L^6(\Omega)^3}\Vert\nabla\varphi_n\Vert_{L^3(\Omega)^3}\Vert\uvec_n\Vert\Vert\wvec\Vert_{D(A)},\\
&\Big|\frac{1}{2}\Big(\rhot'(\varphi_n)\big(\mbox{div}(m(\varphi_n)\nabla\mu_n)-P_n(\mbox{div}(m(\varphi_n)\nabla\mu_n))\big)\uvec_n,\wvec_I\Big)\Big|
\nonumber\\
&\leq \big\Vert\mbox{div}(m(\varphi_n)\nabla\mu_n)\big\Vert_{V'}\big\Vert\rhot'(\varphi_n)\uvec_n\cdot\wvec_I\big\Vert_V\nonumber\\
&\leq c\Vert\nabla\mu_n\Vert\big(\Vert \uvec_n\Vert_{L^6(\Omega)^3}\Vert\wvec_I\Vert_{H^3(\Omega)^3}
+\Vert\nabla\uvec_n\Vert\Vert\wvec_I\Vert_{H^2(\Omega)^3}
+\Vert\nabla\varphi_n\Vert_{L^{10/3}(\Omega)^3}\Vert\uvec_n\Vert_{L^6(\Omega)^3}\Vert\wvec_I\Vert_{H^2(\Omega)^3}\big)\nonumber\\
&\leq c\Vert\nabla\mu_n\Vert\big(\Vert\nabla\uvec_n\Vert
+\Vert\nabla\varphi_n\Vert_{L^{10/3}(\Omega)^3}\Vert\uvec_n\Vert_{L^6(\Omega)^3}\big)\Vert\wvec\Vert_{D(A^{3/2})},\\
&|(\varphi_n\nabla\mu_n,\wvec_I)|\leq\Vert\varphi_n\Vert\Vert\nabla\mu_n\Vert\Vert\wvec_I\Vert_{H^2(\Omega)^3}\leq
\Vert\varphi_n\Vert\Vert\nabla\mu_n\Vert\Vert\wvec\Vert_{D(A)},
\label{rhoutd2}
\end{align}
where \eqref{bdrho1}, (A2) and (A3) have been used.
We now need the following
interpolation
embeddings
\begin{align*}
&L^\infty(0,T;L^2(\Omega))\cap L^2(0,T;H^3(\Omega))\hookrightarrow L^s(0,T;H^{6/s}(\Omega))
\hookrightarrow L^s(0,T; L^{2s/(s-4)}(\Omega)),
\end{align*}
for $4<s\leq\infty$, and
\begin{align*}
&L^\infty(0,T;H^1(\Omega))\cap L^2(0,T;H^2(\Omega))\hookrightarrow L^s(0,T;H^{1+2/s}(\Omega))
\hookrightarrow L^s(0,T; L^{6s/(s-4)}(\Omega)),
\end{align*}
for $4< s\leq\infty$.
 In particular we have
\begin{align*}
&L^\infty(0,T;L^2(\Omega))\cap L^2(0,T;H^3(\Omega))\hookrightarrow L^6(Q),\quad
L^\infty(0,T;H^1(\Omega))\cap L^2(0,T;H^2(\Omega))\hookrightarrow L^{10}(Q).
\end{align*}
These interpolaton embeddings and \eqref{bd1}--\eqref{bd5}, \eqref{bd8} entail the following bounds
\begin{align}
&\Vert\uvec_n\Vert_{L^6(Q)^3}\leq C_\delta,\qquad\Vert\nabla\uvec_n\Vert_{L^{18/5}(Q)^{3\times 3}}\leq C_\delta,
\label{bound1}\\
&\Vert\varphi_n\Vert_{L^{10}(Q)}\leq C_\delta,\qquad\Vert\nabla\varphi_n\Vert_{L^{10/3}(Q)^{3}}\leq C_\delta,\label{bound2}
\end{align}
where the second bound in \eqref{bound1} follows from the bound of
$\nabla\uvec_n$ in $L^s(0,T;H^{(6-s)/s}(\Omega)^{3\times 3})\hookrightarrow L^s(0,T;L^{6s/(5s-12)}(\Omega)^{3\times 3})$
 (for $s>12/5$; take $s=18/5$), and the second bound in \eqref{bound2} follows
 from the bound of $\nabla\varphi_n$ in $L^s(0,T;H^{2/s}(\Omega)^3)\hookrightarrow L^s(0,T;L^{6s/(3s-4)}(\Omega)^3)$
 (for $s>4/3$; take $s=10/3$).
Therefore, by means of estimates \eqref{rhoutd1}--\eqref{rhoutd2} and on account of
the first bound \eqref{bound1} and of the second bound \eqref{bound2},
from \eqref{esttimeder} we then deduce the following estimate (not uniform in $\delta$)
\begin{align}
&\big\Vert\partial_t\widetilde{P}_n\big(\rhot(\varphi_n)\uvec_n\big)\big\Vert_{L^{30/29}(0,T;D(A^{3/2})')}\leq C_\delta.
\label{bd10}
\end{align}

Now, we have
\begin{align*}
\nabla\big(\rhot(\varphi_n)\uvec_n\big)=\rhot(\varphi_n)\nabla\uvec_n+\rhot'(\varphi_n)\nabla\varphi_n\cdot\uvec_n,
\end{align*}
and hence from \eqref{bound1}, \eqref{bound2} we see that $\nabla\big(\rhot(\varphi_n)\uvec_n\big)$ is bounded in $L^{15/7}(Q)^3$ which implies that
\begin{align}
&\Vert\rhot(\varphi_n)\uvec_n\Vert_{L^{15/7}(0,T;W^{1,15/7}(\Omega)^3)}\leq C_\delta.\label{bd12}
\end{align}
To get strong convergence for the sequence of $\varphi_n$, let us first observe that, taking $\psi\in V$
and writing $\psi=\psi_I+\psi_{II}$, where $\psi_I\in\Psi_n$ and $\psi_{II}\in\Psi_n^{\bot}$
(recall that $\psi_I$ and $\psi_{II}$ are orthogonal in all Hilbert spaces $D(B^{r/2})$, for $0\le r\leq 1$), from
\eqref{FGapp2} we have
\begin{align}
&\langle\varphi_n',\psi\rangle_V=(\varphi_n',\psi_I)=-(\nabla\mu_n,\nabla\psi_I)+(\varphi_n\uvec_n,\nabla\psi_I),\nonumber
\end{align}
and, since $p\geq 3$
\begin{align}
&|(\varphi_n\uvec_n,\nabla\psi_I)|\leq\Vert\varphi_n\Vert_{L^{p}(\Omega)}\Vert\uvec_n\Vert_{L^6(\Omega)^3}\Vert\nabla\psi_I\Vert
\leq c\Vert\varphi_n\Vert_{L^{p}(\Omega)}\Vert\nabla\uvec_n\Vert\Vert\psi\Vert_V,\nonumber
\end{align}
whence we deduce that
\begin{align}
&\Vert\varphi_n'\Vert_{L^2(0,T;V')}\leq C.\label{bd11}
\end{align}
Since $\varphi_n$ is bounded in, e.g., $L^2(0,T;H^2(\Omega))$, by Aubin-Lions lemma we then deduce that
\begin{align}
&\varphi_n\to\varphi,\qquad\mbox{strongly in }L^2(0,T;H^{2-\gamma}(\Omega)),\qquad\gamma>0
\label{strong2}
\end{align}
and in particular we have
\begin{align}
&\varphi_n\to\varphi,\qquad\mbox{pointwise a.e. in }Q.\nonumber
\end{align}
Observe also that by Lebesgue's theorem we have
\begin{align}
&\rhot(\varphi_n)\to\rhot(\varphi),\qquad\mbox{strongly in }L^q(Q),\quad\forall q\in [2,\infty),
\label{strong5}
\end{align}
and the same strong convergence holds also for $1/\rhot(\varphi_n)$ to $1/\rhot(\varphi)$.

We now derive strong convergence for the sequence of $\uvec_n$. To this aim notice first
that from \eqref{bd12} we have
\begin{align}
&\big\Vert \widetilde{P}_n\big(\rhot(\varphi_n)\uvec_n\big)\big\Vert_{L^2(0,T;H^1(\Omega)^3)}\leq C_\delta.\label{bd13}
\end{align}
From \eqref{bd10} and \eqref{bd13}, again by means of Aubin-Lions lemma we therefore deduce that
\begin{align*}
&\widetilde{P}_n\big(\rhot(\varphi_n)\uvec_n\big)\to\overline{\rhot(\varphi)\uvec},\qquad\mbox{strongly in }L^2(Q)^3,
\end{align*}
for some $\overline{\rhot(\varphi)\uvec}\in L^2(Q)^3$.
But, from \eqref{strong5} and the weak convergence available for $\uvec_n$ we have
 $\rhot(\varphi_n)\uvec_n\rightharpoonup\rhot(\varphi)\uvec$, weakly in $L^2(Q)^3$ (also in $L^{6-\gamma}(Q)^3$, for $\gamma>0$),
 which implies that $\widetilde{P}_n\big(\rhot(\varphi_n)\uvec_n\big)\rightharpoonup\rhot(\varphi)\uvec$, weakly in $L^2(Q)^3$.
Therefore, we deduce that $\overline{\rhot(\varphi)\uvec}=\rhot(\varphi)\uvec$ and so
\begin{align}
&\widetilde{P}_n\big(\rhot(\varphi_n)\uvec_n\big)\to\rhot(\varphi)\uvec,\qquad\mbox{strongly in }L^2(Q)^3.\label{bd14}
\end{align}
In particular, from \eqref{bd10}, \eqref{bd14} and from \eqref{bd11}, \eqref{weak3} there follows that
up to a subsequence we have
\begin{align}
&\partial_t\widetilde{P}_n\big(\rhot(\varphi_n)\uvec_n\big)\rightharpoonup\big(\rhot(\varphi)\uvec\big)_t,\qquad
\mbox{weakly in }L^{30/29}(0,T;D(A^{3/2})'),\label{weak6}\\
&\varphi_n'\rightharpoonup\varphi_t,\qquad\mbox{weakly in }L^2(0,T;V').\label{weak7}
\end{align}
With \eqref{bd14} and strong convergence for the sequence of $\varphi_n$ at disposal we can
now establish strong convergence for the sequence of $\uvec_n$ by using a classical
argument (cf. \cite{A1}, \cite[Section 2.1]{Lions}), which we report for the reader's convenience.
Indeed, we have
\begin{align}
&\int_0^T\int_\Omega\rhot(\varphi_n)\uvec_n^2=\int_0^T\int_\Omega\widetilde{P}_n\big(\rhot(\varphi_n)\uvec_n\big)\cdot\uvec_n
\to\int_0^T\int_\Omega\rhot(\varphi)\uvec^2,\label{trick}
\end{align}
which means that the $L^2(Q)^3$-norm of $\sqrt{\rhot(\varphi_n)}\uvec_n$ converges to the
$L^2(Q)^3$-norm of $\sqrt{\rhot(\varphi)}\uvec$. Since, thanks to strong convergence for the sequence
of $\varphi_n$, we have also
$\sqrt{\rhot(\varphi_n)}\uvec_n\rightharpoonup\sqrt{\rhot(\varphi_n)}\uvec$, weakly in $L^2(Q)^3$,
then $\sqrt{\rhot(\varphi_n)}\uvec_n\to\sqrt{\rhot(\varphi_n)}\uvec$, strongly in $L^2(Q)^3$.
Finally, notice that due to \eqref{strong5} we have
$\rhot(\varphi_n)^{-1/2}\to\rhot(\varphi)^{-1/2}$, strongly in $L^q(Q)$ (for all $q\in[2,\infty)$) and therefore
we get $\uvec_n\to\uvec$ strongly in $L^{2-\gamma}(Q)^3$ ($\gamma>0$), which implies that, up to a subsequence we have
 $\uvec_n\to\uvec$
pointwise a.e. in $Q$. Since the sequence of $\uvec_n$ is bounded in $L^6(Q)^3$, then
\begin{align}
&\uvec_n\to\uvec,\qquad
\mbox{strongly in } L^{6-\gamma}(Q)^3.\label{strong3}
\end{align}
Furthermore, from \eqref{strong2} and the second \eqref{bound2} we obtain
\begin{align}
&\nabla\varphi_n\to\nabla\varphi,\qquad\mbox{strongly in } L^{10/3-\gamma}(Q)^3.\label{strong4}
\end{align}
Now, the strong convergences \eqref{strong3}, \eqref{strong4}, \eqref{strong5}, together
with the weak convergences \eqref{weak1}--\eqref{weak5}, \eqref{weak6}, \eqref{weak7}
and with the strong convergences $\rhot''(\varphi_n)\to\rhot''(\varphi)$
and $\rhot'(\varphi_n)\to\rhot'(\varphi)$ in $L^q(Q)$, for all $q<\infty$
allow to pass to the limit in the approximate problem \eqref{FGapp1}--\eqref{FGapp5}
and to recover the weak formulation \eqref{wf1step1}, \eqref{wf2step1}.
In particular, observe that we have
\begin{align*}
&\uvec_n\cdot\nabla\varphi_n\to\uvec\cdot\nabla\varphi,\qquad\mbox{strongly in } L^2(Q),
\end{align*}
(also strongly in $L^{15/7-\gamma}(Q)$, for all $\gamma>0$, due to \eqref{strong3}, \eqref{strong4}) and therefore
\begin{align*}
&P_n(\uvec_n\cdot\nabla\varphi_n)\to\uvec\cdot\nabla\varphi,\qquad\mbox{strongly in } L^2(Q).
\end{align*}
Hence, when we pass to the limit the contribution of the seventh term on the left hand side of \eqref{FGapp1}
converges to zero. Moreover, we have
\begin{align*}
&\mbox{div}(m(\varphi_n)\nabla\mu_n)\rightharpoonup\mbox{div}(m(\varphi)\nabla\mu),\qquad\mbox{weakly in }L^2(0,T;V'),
\end{align*}
and therefore also
\begin{align*}
&P_n\big(\mbox{div}(m(\varphi_n)\nabla\mu_n)\big)\rightharpoonup\mbox{div}(m(\varphi)\nabla\mu),\qquad\mbox{weakly in }L^2(0,T;V').
\end{align*}
In addition, for all $\wvec\in D(A^{3/2})$ it is easy to see that we have
\begin{align*}
&\rhot'(\varphi_n)\uvec_n\wvec\to\rhot'(\varphi)\uvec\wvec,\qquad\mbox{strongly in }L^2(0,T;V)
\end{align*}
(also strongly in $L^{15/7-\gamma}(0,T;W^{1,15/7-\gamma}(\Omega))$, due to \eqref{strong3}, \eqref{strong4}), and therefore also
the contribution of the last term on the left hand side of \eqref{FGapp1}
converges to zero when passing to the limit.

We now claim that $\uvec\in C_w([0,T];G_{div})$. To prove this claim, first
 observe that from \eqref{bd1} and \eqref{bd10}
we deduce the boundedness of
\begin{align*}
&\widetilde{P}_n\big(\rhot(\varphi_n)\uvec_n\big)\quad\mbox{in }L^\infty(0,T;G_{div})\mbox { and }\\
&\widetilde{P}_n\big(\rhot(\varphi_n)\uvec_n\big)\quad\mbox{in }W^{1,30/29}(0,T;D(A)')\hookrightarrow C([0,T];D(A)').
\end{align*}
Therefore, thanks to Lemma \ref{Strauss} and on account of \eqref{bd14} we infer
that
\begin{align}
&\zvec:=\rhot(\varphi)\uvec\in C_w([0,T];L^2(\Omega)^d).\nonumber
\end{align}
Moreover, since $\varphi\in C([0,T];H)$, then, by \eqref{bdrho1} we have
$\rhot(\varphi)\in C([0,T];H)$.
Now, let $\uvec$ be the representative, in the equivalence class of $\uvec$, given by $\uvec=\rhot(\varphi)^{-1}\zvec$.
Due to the boundedness of $\rhot(\varphi)$ and noting in particular that we have also $\rhot(\varphi)^{-1}\in C([0,T];H)$,
it is now easy to get the desired claim.

Let us now prove that the initial conditions \eqref{Pb07} are satisfied.
The argument to see that $\varphi(0)=\varphi_0$ holds, by integrating \eqref{FGapp2} in time
between 0 and $t$ and using the strong convergence $\varphi_n\to\varphi$ in $C([0,T];H)$
(which follows from \eqref{bd4} and \eqref{bd11}) is quite standard and we omit the details.
We just give some details on the argument to prove that $\uvec(0)=\uvec_0$.
Let us then take $\wvec\in D(A^{3/2})$
and set $\wvec^N:=\sum_{k=1}^N \alpha_k\wvec_k$, for $N\leq n$, where
$\alpha_k=(\wvec,\wvec_k)$. We multiply \eqref{FGapp1} by $\alpha_k$,
sum on $k$ from $1$ to $N\leq n$, and then integrate the resulting identity
between $0$ and $t$ to get
\begin{align}
&\Big(\rhot(\varphi_n(t))\uvec_n(t),\wvec^N\Big)-
\Big(\rhot(\varphi_{0n})\uvec_{0n},\wvec^N\Big)
-\int_0^t\big(\rhot(\varphi_n)\uvec_n\otimes \uvec_n,D\wvec^N\big)d\tau\nonumber\\
&+2\int_0^t\big(\nu(\varphi_n)D\uvec_n,D\wvec^N\big)d\tau+\delta\int_0^t(A^{3/2}\uvec_n,A^{3/2}\wvec^N)d\tau\nonumber\\
&-\int_0^t\int_\Omega \uvec_n\cdot (\Jvect_n\cdot\nabla)\wvec^N
+\frac{1}{2}\int_0^t\big(\rhot''(\varphi_n) m(\varphi_n)(\nabla\varphi_n\cdot\nabla\mu_n)\uvec_n,\wvec^N\big)d\tau\nonumber\\
&+\frac{1}{2}\int_0^t\Big(\rhot'(\varphi_n)\big(P_n(\uvec_n\cdot\nabla\varphi_n)-\uvec_n\cdot\nabla\varphi_n\big)\uvec_n,\wvec^N\Big)d\tau\nonumber\\
&+\frac{1}{2}\int_0^t\Big(\rhot'(\varphi_n)\big(\mbox{div}(m(\varphi_n)\nabla\mu_n)-P_n(\mbox{div}(m(\varphi_n)\nabla\mu_n))\big)
\uvec_n,\wvec^N\Big)d\tau\nonumber\\
&=-\int_0^t(\varphi_n\nabla\mu_n,\wvec^N)d\tau+\int_0^t\langle\hvec_n,\wvec^N\rangle d\tau\label{idic}
\end{align}
Let us now multiply \eqref{idic} by $\chi\in C^\infty_0(0,T)$ and integrate again in time between 0 and $T$.
We then pass to the limit in the resulting identity first as $n\to\infty$, by using \eqref{bd14} and the weak/strong convergences
obtained above, and then as $N\to\infty$. We then perform a similar computation on the weak formulation \eqref{wf1step1}
(with the same test function $\wvec$) by first integrating it in time between 0 and $t$, by multiplying then
the resulting identity by $\chi$
and integrating again in time between 0 and $T$ .
Comparing the two identities obtained in this way we are led to
\begin{align*}
&\Big(\rhot(\varphi_{0})\uvec_{0},\wvec\Big)\int_0^T\chi(t)dt=\Big(\rhot(\varphi(0))\uvec(0),\wvec\Big)\int_0^T\chi(t)dt,
\end{align*}
which yields $\rhot(\varphi(0))\uvec(0)=\rhot(\varphi_{0})\uvec_{0}$ and therefore $\uvec(0)=\uvec_{0}$, since
$\varphi(0)=\varphi_{0}$.

We now want to show the energy inequality \eqref{enin}.
To this aim we integrate \eqref{FGenid} in time  between 0 and $t$ and get
\begin{align}
&\int_\Omega\frac{1}{2}\rhot(\varphi_n(t))\uvec_n^2(t)+E(\varphi_n(t))+\frac{\delta}{2}\Vert\nabla\varphi_n(t)\Vert^2+
\int_0^t\Big(2\Vert\sqrt{\nu(\varphi_n)}D\uvec_n\Vert^2+\delta\Vert A^{3/2}\uvec_n\Vert^2\nonumber\\
&+\Vert\sqrt{m(\varphi_n)}\nabla\mu_n\Vert^2\Big)d\tau=
\int_\Omega\frac{1}{2}\rhot(\varphi_{0n})\uvec_{0n}^2+E(\varphi_{0n})+\frac{\delta}{2}\Vert\nabla\varphi_{0n}\Vert^2+
\int_0^t\langle\hvec_n,\uvec_n\rangle_{V_{div}} d\tau,
\label{FGenidint}
\end{align}
for all $t\in[0,T]$.
Then, we pass to the limit in \eqref{FGenidint} taking the weak/strong convergences above into account
and using the weak lower semicontinuity of norms.
In particular, as far as the term containing the variable viscosity is concerned, we observe that
$D\uvec_n\rightharpoonup D\uvec$ weakly in $L^{18/5}(Q)^{3\times 3}$,
while
$\sqrt{\nu(\varphi_n)} \rightharpoonup \sqrt{\nu(\varphi)}$ strongly in $L^q(Q)$ for all $2\leq q<\infty$.
Hence $\sqrt{\nu(\varphi_n)}D\uvec_n \rightharpoonup \sqrt{\nu(\varphi)}D\uvec$ weakly in $L^{18/5-\gamma}(Q)^{3\times 3}$
($\gamma>0$) and therefore also weakly in $L^2(Q)^{3\times 3}$. By also employing Lemma \ref{simplelem} to pass to the liminf
in the term containing the variable mobility we deduce
\begin{align}
&\int_0^t 2\Vert\sqrt{\nu(\varphi)}D\uvec\Vert^2 d\tau\leq\liminf_{n\to\infty}\int_0^t 2\Vert\sqrt{\nu(\varphi_n)}D\uvec_n\Vert^2 d\tau,
\nonumber\\
&\int_0^t \Vert\sqrt{m(\varphi)}\nabla\mu\Vert^2 d\tau\leq\liminf_{n\to\infty}\int_0^t \Vert\sqrt{m(\varphi_n)}\nabla\mu_n\Vert^2 d\tau.
\nonumber
\end{align}
Concerning the first term on the left hand side of \eqref{FGenidint}, in view
of \eqref{strong3} and \eqref{strong5} which imply that
\begin{align}
&\rhot(\varphi_n)\uvec_n^2\to\rhot(\varphi)\uvec^2,\qquad\mbox{ strongly in }L^{3-\gamma}(Q),\label{strong6}
\end{align}
we have that this term converges for almost all $t\in(0,T)$ to the first term on the left hand side of \eqref{enin}.
The fact that the term in the $L^2-$norm of $\nabla\varphi_n$ converges, for almost all $t\in(0,T)$, to the third
term on the left hand side of \eqref{enin} is a consequence of \eqref{strong4}. The passage to the limit in the other terms
in \eqref{FGenidint} is straightforward and this concludes the proof of \eqref{enin}, which holds
for almost all $t\in(0,T)$.

Finally, we deduce an auxiliary energy inequality which shall turn out to be useful
to deduce the energy inequality \eqref{eninPbor}.
Let us multiply \eqref{FGenid} by $\omega\in W^{1,1}(0,T)$, with $\omega(T)=0$ and $\omega\geq 0$. We obtain
\begin{align}
&\Big(\int_\Omega\frac{1}{2}\rhot(\varphi_{0n})\uvec_{0n}^2+E(\varphi_{0n})+\frac{\delta}{2}\Vert\nabla\varphi_{0n}\Vert^2\Big)\omega(0)+
\int_0^T\Big(\int_\Omega\frac{1}{2}\rhot(\varphi_n)\uvec_n^2+E(\varphi_n)+\frac{\delta}{2}\Vert\nabla\varphi_n\Vert^2\Big)\omega'(\tau)d\tau\nonumber\\
&=\int_0^T\Big(2\Vert\sqrt{\nu(\varphi_n)}D\uvec_n\Vert^2+\delta\Vert A^{3/2}\uvec_n\Vert^2
+\Vert\sqrt{m(\varphi_n)}\nabla\mu_n\Vert^2\Big)\omega(\tau)d\tau-\int_0^T\langle\hvec_n,\uvec_n\rangle_{V_{div}}
\omega(\tau)d\tau.
\label{FGenidomega}
\end{align}
In order to pass to the limit in the term containing the functional $E(\varphi_n)$
on the left hand side of \eqref{FGenidomega} we need to prove that
\begin{align}
&\int_0^T\omega'(\tau)d\tau\int_\Omega F(\varphi_n)\to\int_0^T\omega'(\tau)d\tau\int_\Omega F(\varphi).
\label{claim}
\end{align}
To this aim observe that, thanks to the following compact and continuous embeddings
\begin{align*}
&L^2(0,T;H^2(\Omega))\cap H^1(0,T;V')\hookrightarrow\hookrightarrow L^2(0,T;H^s(\Omega))\hookrightarrow L^2(0,T;C(\overline{\Omega})),
\qquad 3/2<s<2,
\end{align*}
and to the bounds \eqref{bd8} and \eqref{bd11} we have
\begin{align*}
&\varphi_n(\tau)\to\varphi(\tau),\qquad \mbox{ in }\:\: C(\overline{\Omega}),
\end{align*}
for almost all $\tau\in(0,T)$. On the other hand, due to the energy identity \eqref{FGenidint}
the sequence of integrals $\int_\Omega F(\varphi_n(\tau))$ is uniformly bounded with respect
to $n\in\mathbb{N}$ and for a.a. $\tau\in(0,T)$ (cf. \eqref{bd6}).
Hence, \eqref{claim} follows immediately by applying Lebesgue's theorem.


Passing to the limit in \eqref{FGenidomega} and employing weak/strong convergences
for $\uvec_n$, $\varphi_n$, $\mu_n$ (recall, in particular, \eqref{strong4} and \eqref{strong6}),
and weak lower semicontinuity of norms
in the same fashion as done for the proof of \eqref{enin} we get
\begin{align}
&\Big(\int_\Omega\frac{1}{2}\rhot(\varphi_{0})\uvec_{0}^2+E(\varphi_{0})+\frac{\delta}{2}\Vert\nabla\varphi_{0}\Vert^2\Big)\omega(0)+
\int_0^T\Big(\int_\Omega\frac{1}{2}\rhot(\varphi)\uvec^2+E(\varphi)+\frac{\delta}{2}\Vert\nabla\varphi\Vert^2\Big)\omega'(\tau)d\tau\nonumber\\
&\geq\int_0^T\Big(2\Vert\sqrt{\nu(\varphi)}D\uvec\Vert^2+\delta\Vert A^{3/2}\uvec\Vert^2
+\Vert\sqrt{m(\varphi)}\nabla\mu\Vert^2\Big)\omega(\tau)d\tau-\int_0^T\langle\hvec,\uvec\rangle_{V_{div}}
\omega(\tau)d\tau,
\label{eninomega}
\end{align}
for all $\omega\in W^{1,1}(0,T)$, with $\omega(T)=0$ and $\omega\geq 0$.

\end{proof}

\subsection{Step II. Limit as $\e\to 0$.}\label{lim-eps}

Next, we consider problem \eqref{Pb11}--\eqref{Pb17}
where now the potential $F$ is singular and
$\delta>0$ is still fixed (in this subsection we shall denote the initial
datum for $\varphi$ simply by $\varphi_0$, instead of $\varphi_{0\delta}$).
We aim to prove that this problem admits a weak solution by approximating it with a sequence
of problems $P_{\epsilon,\delta}$ of the form \eqref{Pb01}--\eqref{Pb07} with regular potentials $F_\e$.
More precisely, we prove the following
\begin{lem}
\label{existence2}
Let assumptions (A1)--(A9) be satisfied for some fixed integer $p\geq 3$. Let $\uvec_0\in G_{div}$, $\varphi_0\in V\cap L^\infty(\Omega)$
 such that $F(\varphi_0)\in L^1(\Omega)$
and $|\overline{\varphi}_0|<1$. Then, for every $T>0$ Problem \eqref{Pb11}--\eqref{Pb17}
admits a weak solution $[\uvec,\varphi]$ on $[0,T]$ corresponding to $[\uvec_0,\varphi_0]$ such that
\begin{align}
&\uvec\in C_w([0,T];G_{div})\cap L^2(0,T;D(A^{3/2})),\label{regp1Pb1}\\
&\varphi\in L^\infty(0,T;V\cap L^p(\Omega))\cap L^2(0,T;H^2(\Omega)),\label{regp2Pb1}\\
&\mu\in L^2(0,T;V),\label{regp3Pb1}
\end{align}
satisfying the bound
\begin{align}
&|\varphi(x,t)|<1,\qquad\mbox{for a.e. }(x,t)\in Q,\label{pointbdphi}
\end{align}
 and satisfying the energy inequality \eqref{enin} (with $\rho$ in place of $\rhot$)
for almost all $t\in(0,T)$.
\end{lem}

\begin{proof}

We consider Problem P$_{\epsilon,\delta}$ consisting in \eqref{Pb01}--\eqref{Pb07} where the regular potential $F$
is taken equal to $F_\epsilon$ given by (see \cite{FG2} and \cite{FGR})
$$\Fe=\Fie+F_{2\e},$$
where $\Fie$ and $F_{2\e}$ are defined by
\begin{equation}
\Fie^{(p)}(s)=\left\{\begin{array}{lll}
F_1^{(p)}(1-\e),\qquad s\geq 1-\e\\
F_1^{(p)}(s),\qquad|s|\leq 1-\e\\
F_1^{(p)}(-1+\e),\qquad s\leq -1+\e,
\end{array}\right.
\label{approxpot1}
\end{equation}
\begin{align}
& F_{2\e}^{''}(s)=\left\{\begin{array}{lll}
F_2^{''}(1-\e),\qquad s\geq 1-\e\\
F_2^{''}(s),\qquad|s|\leq 1-\e\\
F_2^{''}(-1+\e),\qquad s\leq -1+\e,
\end{array}\right.
 \label{approxpot2}
\end{align}
and $\Fie(0)=F_1(0)$,
$\Fie'(0)=F_1'(0)$,$\dots$, $\Fie^{(p-1)}(0)=F_1^{(p-1)}(0)$,
and $F_{2\e}(0)=F_2(0)$, $F_{2\e}'(0)=F_2'(0)$.
Recalling \cite[Lemma1, Lemma 2]{FG2} and \cite[Proof of Theorem 2]{FGR}, there exist
two constants $C_p$ and $D_p$, depending on $p$ but independent of $\e$, and there exists $\e_0>0$
such that
\begin{align}
&\Fe(s)\geq C_p|s|^p-D_p,\qquad\forall s\in\mathbb{R},\quad\forall\e\in(0,\e_0],\label{auxlem1}
\end{align}
and
\begin{align}
&\Fe''(s)+a(x)\geq c_0,\qquad\forall s\in\mathbb{R},\quad\mbox{a.e. }x\in\Omega,\quad\forall\e\in(0,\e_0].\label{auxlem2}
\end{align}
Now, thanks to Lemma \ref{existence1} we know that Problem P$_{\e,\delta}$ admits a weak solution $[\uvec_\e,\varphi_\e]$
with the regularity properties \eqref{regp1Pb0}--\eqref{regp3Pb0} satisfying the energy inequality \eqref{enin}
\begin{align}
&\int_\Omega\frac{1}{2}\rhot(\varphi_\e)\uvec_\e^2+E_\e(\varphi_\e)+\frac{\delta}{2}\Vert\nabla\varphi_\e\Vert^2+
2\int_0^t\Vert\sqrt{\nu(\varphi_\e)}D\uvec_\e\Vert^2 d\tau+\delta\int_0^t\Vert A^{3/2}\uvec_\e\Vert^2 d\tau\nonumber\\
&+\int_0^t\big\Vert\sqrt{m(\varphi_\e)}\nabla\mu_\e\big\Vert^2 d\tau
\leq\int_\Omega\frac{1}{2}\rhot(\varphi_0)\uvec_0^2+E_\e(\varphi_0)+\frac{\delta}{2}\Vert\nabla\varphi_0\Vert^2
+\int_0^t\langle \hvec,\uvec_\e\rangle_{V_{div}} d\tau,
\label{eninepsilon}
\end{align}
for almost all $t\in(0,T)$, where
\begin{align*}
&E_\e(\varphi):=\frac{1}{2}\Vert\sqrt{a}\varphi\Vert^2-\frac{1}{2}(\varphi,J\ast\varphi)+\int_\Omega F_\e(\varphi)\nonumber\\
&=\frac{1}{4}\int_\Omega\int_\Omega J(x-y)\big(\varphi(x)-\varphi(y)\big)^2 dxdy+\int_\Omega F_\e(\varphi).
\end{align*}
Notice that in the weak formulation of \eqref{Pb01} we don't have additional auxiliary terms, like the last
two terms on the left hand side of \eqref{FGapp1}.

Since $\delta>0$ is fixed, starting from \eqref{eninepsilon}, splitting the last term on the right hand side
by taking (A9) into account, and arguing as in the Faedo-Galerkin scheme of the proof of Lemma \ref{existence1}
by employing assumptions
(A1)--(A3), \eqref{bdrho1},
 \eqref{auxlem1}, \eqref{auxlem2} and (A6) which implies that there exists $\e_0>0$ such that
\begin{align}
&F_{1\e}(s)\leq F_1(s),\qquad\forall s\in (-1,1),\quad\forall\e\in(0,\e_0],\label{boundF}
\end{align}
we still deduce estimates \eqref{bd1}--\eqref{bd5} and
\eqref{bd7}, \eqref{bd8}, \eqref{bd11} for $\uvec_\e$, $\varphi_\e$, and $\mu_\e$, namely
\begin{align}
&\Vert\uvec_\e\Vert_{L^\infty(0,T;G_{div})\cap L^2(0,T;V_{div})}\leq C,\label{bd1eps}\\
\delta^{1/2}&\Vert\uvec_\e\Vert_{L^2(0,T;D(A^{3/2}))}\leq C,\label{bd2eps}\\
&\Vert\varphi_\e\Vert_{L^\infty(0,T;L^p(\Omega))\cap L^2(0,T;V) }\leq C,\label{bd3eps}\\
\delta^{1/2}&\Vert\varphi_\e\Vert_{L^\infty(0,T;V)\cap L^2(0,T;H^2(\Omega))}\leq C,\label{bd4eps}\\
&\Vert\nabla\mu_\e\Vert_{L^2(0,T;H)}\leq C ,\label{bd5eps}\\
&\Vert\varphi_\e'\Vert_{L^2(0,T;V')}\leq C,\label{bd9eps}
\end{align}
which are now uniform in $\e$. Here, all constants $C$ have the same kind of dependencies
on the data $\uvec_0$, $\varphi_0$, $\hvec$ and on the parameters of the problem as in Step I.
Moreover, the estimate for the time derivatives $(\rhot(\varphi_\e)\uvec_\e\big)_t$
which corresponds to \eqref{bd10}
and obtained by comparison in the weak formulation of \eqref{Pb01} (cf. \eqref{rhoutd1}--\eqref{rhoutd2})
 now becomes
\begin{align}
&\big\Vert\big(\rhot(\varphi_\e)\uvec_\e\big)_t\big\Vert_{L^{30/29}(0,T;D(A^{3/2})')}\leq C_\delta.\label{bd8eps}
\end{align}
We also need an estimate for the sequence of $\mu_\e$ and in particular we need to control the sequence of averages
$\overline{\mu}_\e$. To this aim we notice that equation \eqref{Pb03} can be written in the form
\begin{align}
&\varphi_\e'+\uvec_\e\cdot\nabla\varphi_\e=-\mathcal{B}_{\varphi_\e}\mu_\e.\label{Pb03eps}
\end{align}
Test \eqref{Pb03eps} by $\mathcal{N}_{\varphi_\e}\big(\Fe'(\varphi_\e)-\overline{\Fe'(\varphi_\e)}\big)$. On account of
\eqref{mathN1} and \eqref{mathN2} we obtain
\begin{align}
&\big\langle\Fe'(\phie)-\overline{\Fe'(\phie)},\mathcal{N}_{\phie}\phie'\big\rangle_V
+\big\langle\Fe'(\phie)-\overline{\Fe'(\phie)},\mathcal{N}_{\phie}(\ue\cdot\nabla\phie)\big\rangle_V\nonumber\\
&=-\big\langle\Fe'(\phie)-\overline{\Fe'(\phie)},\mue\big\rangle_V,
\label{Pe1tested}
\end{align}
and since the elements in all dualities of this identity belong to $H$ (the fact that
$\Fe'(\phie)\in H$ follows from comparison in the expression
for $\mue=a\phie-J\ast\phie+\Fe'(\phie)-\delta\Delta\phie$), then
in \eqref{Pe1tested} we can replace all dualities in $V$
with scalar products in $H$.
Now we have
\begin{align}
&\big(\Fe'(\phie)-\overline{\Fe'(\phie)},\mue\big)=
\big(\Fe'(\phie)-\overline{\Fe'(\phie)},a\phie-J\ast\phie+\Fe'(\phie)-\overline{\Fe'(\phie)}-\delta\Delta\phie\big)
\nonumber\\
&\geq\frac{1}{2}\Vert\Fe'(\phie)-\overline{\Fe'(\phie)}\Vert^2-\frac{1}{2}\Vert a\phie-J\ast\phie\Vert^2
+\delta\int_\Omega\Fe''(\phie)|\nabla\phie|^2 \nonumber\\
&\geq\frac{1}{2}\Vert\Fe'(\phie)-\overline{\Fe'(\phie)}\Vert^2-C_J\Vert\phie\Vert^2
+c_0 \delta\Vert\nabla\phie\Vert^2-a_\infty\delta\Vert\nabla\phie\Vert^2.
\label{FeFe'}
\end{align}
Therefore, by combining
\eqref{Pe1tested} with \eqref{FeFe'} we deduce
\begin{align}
&\Vert\Fe'(\phie)-\overline{\Fe'(\phie)}\Vert\leq c(\Vert\mathcal{N}_{\phie}\phie'\Vert
+\Vert\mathcal{N}_{\phie}(\ue\cdot\nabla\phie)\Vert+\sqrt{\delta}\Vert\phie\Vert_V)\nonumber\\
&\leq c(\Vert\phie'\Vert_{V'}+\Vert\ue\cdot\nabla\phie\Vert_{V'}+\sqrt{\delta}\Vert\phie\Vert_V),
\label{FeFe'2}
\end{align}
where \eqref{mathN} has been taken into account as well.
Notice also that we have
\begin{align*}
&|(\ue\cdot\nabla\phie,\psi)|=|(\ue\phie,\nabla\psi)|\leq c\Vert\nabla\uvec_\e\Vert\Vert\phie\Vert_{L^p(\Omega)}\Vert\psi\Vert_V,
\end{align*}
for all $\psi\in V$, which yields, on account of \eqref{bd1eps} and \eqref{bd3eps}
\begin{align}
&\Vert\ue\cdot\nabla\phie\Vert_{L^2(0,T;V')}\leq C.\label{bd10eps}
\end{align}
Hence, due to \eqref{bd9eps}, \eqref{bd10eps} and \eqref{bd3eps} from \eqref{FeFe'2} we obtain the uniform in $\delta$ and $\e$ bound
\begin{align}
&\Vert\Fe'(\phie)-\overline{\Fe'(\phie)}\Vert_{L^2(0,T;H)}\leq C.\label{bd11eps}
\end{align}
With \eqref{bd11eps} available, by employing the condition $|\overline{\varphi}_0|<1$
and the bound
\begin{align}
&|F_{1\e}'(s)|\leq |F_1'(s)|,\qquad\forall s\in(-1,1),\quad\forall \e\in(0,\e_0],\label{boundF'}
\end{align}
for some $\e_0>0$, which is ensured by (A5), (A6) and (A8) (see \cite[Proof of Theorem 1]{FG2} for details),
we can now apply an argument devised by Kenmochi et al. \cite{KNP} (see also \cite{CGGS} and \cite[Proof of Theorem 1]{FG2})
and deduce the following control
\begin{align}
&\Vert\Fe'(\phie)\Vert_{L^2(0,T;L^1(\Omega))}\leq C(\overline{\varphi}_0).\label{bd12eps}
\end{align}
Since
\begin{align*}
&\int_\Omega\mu_\e=\int_\Omega(a\phie-J\ast\phie+F_\e'(\phie)-\delta\Delta\phie)=\int_\Omega F_\e'(\phie),
\end{align*}
then we have $\Vert\overline{\mu}_\e\Vert_{L^2(0,T)}\leq C$ and therefore, by Poincar\'{e}-Wirtinger
inequality we get
\begin{align}
\Vert\mu_\e\Vert_{L^2(0,T;V)}\leq C.\label{bd13eps}
\end{align}
Like in the Faedo-Galerkin scheme ($\delta>0$ here is still fixed),
from estimates \eqref{bd1eps}--\eqref{bd5eps} and \eqref{bd13eps}
we deduce that there exist $\uvec$, $\varphi$ and $\mu$
such that, up to a subsequence, we have
\begin{align}
&\uvec_\e\rightharpoonup \uvec,\qquad\mbox{weakly$^\ast$ in } L^\infty(0,T;G_{div}),\label{weak1eps}\\
&\uvec_\e\rightharpoonup \uvec,\qquad\mbox{weakly in } L^2(0,T;D(A^{3/2})),\label{weak2eps}\\
&\varphi_\e\rightharpoonup \varphi,\qquad\mbox{weakly$^\ast$ in } L^\infty(0,T;V\cap L^p(\Omega)),\label{weak3eps}\\
&\varphi_\e\rightharpoonup \varphi,\qquad\mbox{weakly in } L^2(0,T;H^2(\Omega)),\label{weak4eps}\\
&\mu_\e\rightharpoonup \mu,\qquad\mbox{weakly in } L^2(0,T;V).\label{weak5eps}
\end{align}
Furthermore, by using the bound \eqref{bd10} and the bound (cf. \eqref{bd12})
\begin{align}
&\Vert\rhot(\phie)\uvec_\e\Vert_{L^{15/7}(0,T;W^{1,15/7}(\Omega)^3)}\leq C_\delta,\label{bd14eps}
\end{align}
as well as the bound \eqref{bd9eps}
and arguing as in the Faedo-Galerkin scheme of Step I we can again deduce the strong convergences
\begin{align}
&\phie\to\varphi,\qquad\mbox{strongly in }L^2(0,T;H^{2-\gamma}(\Omega)),\qquad\gamma>0\label{strong6eps},\\
&\uvec_\e\to\uvec,\qquad
\mbox{strongly in } L^{6-\gamma}(Q)^3,\label{strong3eps}\\
&\nabla\varphi_\e\to\nabla\varphi,\qquad\mbox{strongly in } L^{10/3-\gamma}(Q)^3,\label{strong4eps}
\end{align}
as well as
\begin{align}
&\rhot''(\phie)\to\rhot''(\varphi),\qquad m(\phie)\to m(\varphi),
\qquad\mbox{strongly in }L^q(Q),\quad\forall q\in[2,\infty).\label{strong5eps}
\end{align}
Notice that now, in order to deduce strong convergence for $\uvec_\e$ (and therefore \eqref{strong3eps} and
\eqref{strong4eps}) we do not need to
employ the trick used in Step I (see \eqref{trick}). In particular, \eqref{bd8eps}, \eqref{bd14eps}
and Aubin-Lions lemma
 entail strong convergence for $\rhot(\phie)\uvec_\e$ in $L^{15/7}(Q)^3$
 (notice that $W^{1,15/7}(\Omega)^3\hookrightarrow\hookrightarrow L^{15/7}(\Omega)^3$ and that
$L^{15/7}(\Omega)^3\hookrightarrow D(A^{3/2})'$, since $D(A^{3/2})\hookrightarrow L^{15/8}(\Omega)^3$).

In order to pass to the limit in the variational formulation of Problem P$_{\e,\delta}$
and hence to prove that $[\uvec,\varphi]$ is a weak solution to Problem \eqref{Pb11}--\eqref{Pb17}
we need to show that the limit $\varphi$ satisfies the condition $|\varphi|<1$ a.e. in $Q=\Omega\times(0,T)$.
This can be done exactly as in \cite[Proof of Theorem 1]{FG2} by adapting an argument devised in \cite{DD}
(see also \cite{EL}). We recall that this argument is based only on the use of: (i) estimate
 $\Vert \Fe'(\phie)\Vert_{L^1(Q)}\leq C(\overline{\varphi}_0)$ (cf. \eqref{bd12eps}), (ii)
 the pointwise convergence $\phie\to\varphi$ a.e. in $Q$ (cf. \eqref{strong6eps}) and
 (iii) the fact that $F'(s)\to\pm\infty$, as $s\to \pm 1$ (cf. (A8)).
Hence \eqref{pointbdphi} follows.

From this bound, from the pointwise convergence of $\phie$ to $\varphi$ in $Q$ and from the fact that $\Fe'\to F'$
uniformly on every compact interval included in $(-1,1)$ we infer that
\begin{align}
&\Fe'(\phie)\to F'(\varphi)\qquad\mbox{a.e. in }Q.\label{pointwF'}
\end{align}



We are now ready to pass to the limit in the variational formulation \eqref{wf1step1}, \eqref{wf2step1}
of Problem P$_{\e,\delta}$ (with $\delta>0$
fixed) as $\e\to0$, in order to recover the variational formulation of Problem \eqref{Pb11}--\eqref{Pb17}
(still with test functions $\wvec\in D(A^{3/2})$ and $\psi\in V$ for \eqref{Pb11} and \eqref{Pb13}, respectively).
This passage to the limit can be carried out
in the same fashion as done in the proof of Lemma \ref{existence1} (except for the last two terms on the left hand side
of \eqref{FGapp1} which are not present in the weak formulation of Problem P$_{\e,\delta}$),
by employing convergences \eqref{weak1eps}--\eqref{weak5eps}, \eqref{strong6eps}--\eqref{strong5eps}, \eqref{pointwF'}.

Moreover, in this case, as far as the artificial term in the momentum balance equation
is concerned, we observe that its contribution vanishes in the limit.
Indeed,
\eqref{strong3eps}, \eqref{strong4eps}, \eqref{strong5eps}, and \eqref{weak5eps} entail
\begin{align}
&\int_0^T\Big(\frac{1}{2}\rhot''(\phie)m(\phie)(\nabla\phie\cdot\nabla\mue)\uvec_\e,\wvec\Big)\chi(t)dt
\to\int_0^T\Big(\frac{1}{2}\rhot''(\varphi)m(\varphi)(\nabla\varphi\cdot\nabla\mu)\uvec,\wvec\Big)\chi(t)dt=0,\nonumber
\end{align}
as $\e\to 0$, for all test functions $\wvec\in D(A^{3/2})$ and $\chi\in C^\infty_0(0,T)$, where the last identity is due to
\eqref{pointbdphi} and to the fact that $\rhot''(s)=0$ for all $s\in[-1,1]$.
Hence, the weak formulation of \eqref{Pb11}--\eqref{Pb17} is obtained in the limit
with the function $\rho$ in place of $\rhot$, since we have proven that $|\varphi|<1$
and on account of $\rhot(s)=\rho(s)$ for all $s\in(-1,1)$.

The arguments to prove that $\uvec\in C_w([0,T];G_{div})$, that
the limit $[\uvec,\varphi]$ attains the initial value $[\uvec_0,\varphi_0]$
and that the energy inequality \eqref{enin} is satisfied for almost all $t\in(0,T)$
are the same as in Step I and we omit the details.

Finally, we can derive the auxiliary energy inequality \eqref{eninomega} satisfied by the limit
$[\uvec,\varphi]$ (with $\rho$ in place of $\rhot$) for all $\omega\in W^{1,1}(0,T)$ with $\omega(T)=0$ and $\omega\geq 0$.
Indeed, from Step I we can write it first for every $\e-$approximate solution $[\uvec_\e,\varphi_\e]$
\begin{align}
&\Big(\int_\Omega\frac{1}{2}\rhot(\varphi_{0})\uvec_{0}^2+E(\varphi_{0})+\frac{\delta}{2}\Vert\nabla\varphi_{0}\Vert^2\Big)\omega(0)+
\int_0^T\Big(\int_\Omega\frac{1}{2}\rhot(\phie)\uvec_\e^2+E_\e(\varphi_\e)+\frac{\delta}{2}\Vert\nabla\phie\Vert^2\Big)\omega'(\tau)d\tau\nonumber\\
&\geq\int_0^T\Big(2\Vert\sqrt{\nu(\phie)}D\uvec_\e\Vert^2+\delta\Vert A^{3/2}\uvec_\e\Vert^2
+\Vert\sqrt{m(\phie)}\nabla\mue\Vert^2\Big)\omega(\tau)d\tau-\int_0^T\langle\hvec,\uvec_\e\rangle_{V_{div}}
\omega(\tau)d\tau.
\label{eninomegaeps}
\end{align}
As is Step I, in order to pass to the limit in the term containing the functional $E_\e(\phie)$ we need
to show that
\begin{align}
&\int_0^T\omega'(\tau)d\tau\int_\Omega F_\e(\phie)\to\int_0^T\omega'(\tau)d\tau\int_\Omega F(\varphi).\label{claim2}
\end{align}
This convergence can now be established in the following way. First, introduce the function $G_\e$ defined by
\begin{align}
&G_\e(s)=F_\e(s)+\frac{a_\infty}{2}s^2,\label{Gdef}
\end{align}
with $a_\infty=\Vert a\Vert_{L^\infty(\Omega)}$ and, observe that owing to \eqref{auxlem2}, $G_\e$ is convex on $\mathbb{R}$.
Hence, we have
\begin{align}
&\int_\Omega G_\e(\phie)\leq\int_\Omega G_\e(\varphi)+\int_\Omega G_\e'(\phie)(\phie-\varphi),\qquad
 \int_\Omega G_\e(\varphi)\leq \int_\Omega G_\e(\phie)+\int_\Omega G_\e'(\varphi)(\varphi-\phie)\label{convexity}
\end{align}
Introduce the sets $I^\omega_+:=\{t\in(0,T):\omega'(t)\geq 0\}$ and $I^\omega_-:=\{t\in(0,T):\omega'(t)< 0\}$. Then,
multiply \eqref{convexity}$_1$ by $\omega'\chi_{I^\omega_+}$ and \eqref{convexity}$_2$ by $\omega'\chi_{I^\omega_-}$,
where $\chi_{I^\omega_{\pm}}$ are the characteristic functions of the sets $I^\omega_{\pm}$. Integrating in time between 0 and $T$ and
summing the resulting
inequalities we obtain
\begin{align}
&\int_0^T\omega'(\tau)d\tau\int_\Omega G_\e(\phie)\leq\int_0^T\omega'(\tau)d\tau\int_\Omega G_\e(\varphi)
+\int_{I^\omega_+}\omega'(\tau)d\tau\int_\Omega G_\e'(\phie)(\phie-\varphi)\nonumber\\
&+\int_{I^\omega_-}\omega'(\tau)d\tau\int_\Omega G_\e'(\varphi)(\phie-\varphi)\leq
\int_0^T\omega'(\tau)d\tau\int_\Omega G_\e(\varphi)\nonumber\\
&+C_\omega\big(\Vert G_\e'(\phie)\Vert_{L^2(0,T;H)}+\Vert G'(\varphi)\Vert_{L^2(0,T;H)}\big)
\Vert\phie-\varphi\Vert_{L^2(0,T;H)},\label{intermineq}
\end{align}
where, in the last inequality we have exploited the bound \eqref{boundF'}, which in particular
implies that $|G_\e'(s)|\leq |G'(s)|$ for all $s\in(-1,1)$ and for every $\e\in(0,\e_0]$.
Therefore, on account of the estimate $\Vert F_\e'(\phie)\Vert_{L^2(0,T;H)}\leq C$,
which follows from \eqref{bd11eps} and from the bound $\Vert\overline{\mu}_\e\Vert_{L^2(0,T)}\leq C$,
and which implies that $\Vert G_\e'(\phie)\Vert_{L^2(0,T;H)}\leq C$ and thanks to the strong
convergence $\phie\to\varphi$ in $L^2(Q)$, we have that the second term on the right hand side of
\eqref{intermineq} converges to zero as $\e\to0$ and therefore we deduce
\begin{align}
&\limsup_{\e\to 0}\int_0^T\omega'(\tau)d\tau\int_\Omega G_\e(\phie)\leq\lim_{\e\to 0}\int_0^T\omega'(\tau)d\tau\int_\Omega G_\e(\varphi)
=\int_0^T\omega'(\tau)d\tau\int_\Omega G(\varphi),\label{limsup}
\end{align}
where the last equality follows from Lebesgue's theorem (use \eqref{boundF}
and the fact that $|\varphi|<1$ a.e. in $Q$ and that $F_\e(s)\to F(s)$ pointwise for all $s\in(-1,1)$).
On the other hand, thanks to Fatou's lemma and to the pointwise convergence
$F_\e(\phie)\to F(\varphi)$, a.e. in $Q$ we also have
\begin{align}
&\int_0^T\omega'(\tau)d\tau\int_\Omega G(\varphi)\leq\liminf_{\e\to 0}\int_0^T\omega'(\tau)d\tau\int_\Omega G_\e(\phie).
\label{liminf}
\end{align}
From \eqref{limsup} and \eqref{liminf}, on account of the definition \eqref{Gdef} of $G_\e$,
 we get \eqref{claim2}.

We can now pass to the limit in \eqref{eninomegaeps} in the same way as done in Step I, taking
into account the strong/weak convergences for $\uvec_\e$, $\phie$ and $\mu_\e$ . In particular, the strong convergence
$\rhot(\phie)\uvec_\e^2\to\rho(\varphi)\uvec^2$ in $L^{3-\gamma}(Q)$ still holds,
and we take into account \eqref{strong4eps} as well.
Therefore, we have proven that \eqref{eninomega} (with $\rho$ in place of $\rhot$) is satisfied 
and the proof of Lemma \ref{existence2} is concluded.

\end{proof}

\subsection{Step III. Limit as $\delta\to 0$ and end of the proof.}\label{lim-delta}

We now want to pass to the limit in problem \eqref{Pb11}--\eqref{Pb17}
in order to prove that the original problem, i.e.
\begin{align}
&(\rho\uvec)_t+\mbox{div}(\rho\uvec\otimes\uvec)-2\mbox{div}\big(\nu(\varphi)D\uvec\big)+\nabla\pi+\mbox{div}(\uvec\otimes\Jvect)
=\mu\nabla\varphi+\hvec,\label{Pb21}\\
&\mbox{div}(\uvec)=0,\label{P22}\\
&\varphi_t+\uvec\cdot\nabla\varphi=\mbox{div}(m(\varphi)\nabla\mu),\label{Pb23}\\
&\mu=a\varphi-J\ast\varphi+F'(\varphi),\label{Pb24}\\
&\Jvect:=-\beta m(\varphi)\nabla\mu,\label{Pb25}\\
&\uvec=0,\qquad\frac{\partial\mu}{\partial\nvec}=0,\qquad\mbox{on }\partial\Omega,\label{Pb26}\\
&\uvec(0)=\uvec_0,\qquad\varphi(0)=\varphi_0,\label{Pb27}
\end{align}
admits a weak solution. To this aim,
we approximate problem \eqref{Pb21}--\eqref{Pb27}
by a sequence of problems P$_\delta$ given by \eqref{Pb11}--\eqref{Pb17} with initial data $\uvec_0$ and
$\varphi_{0\delta}$, where $\varphi_{0\delta}$ are chosen according with the following
\begin{lem}
\label{lemaux}
Given $\varphi_0\in L^\infty(\Omega)$ with $F(\varphi_0)\in L^1(\Omega)$ and $|\overline{\varphi}_0|<1$,
there exists a sequence $\{\varphi_{0\delta}\}\subset D(B)$
with $F(\varphi_{0\delta})\in L^1(\Omega)$ and $|\overline{\varphi}_{0\delta}|<1$ such that
\begin{align}
&\delta\Vert\nabla\varphi_{0\delta}\Vert^2\to 0\qquad\mbox{and}\qquad\varphi_{0\delta}\to\varphi_0\quad\mbox{in }H,
\qquad\mbox{as }\delta\to 0.\label{auxlem}
\end{align}
\end{lem}
\begin{proof}
Take $\varphi_{0\delta}\in D(B)$ given by
\begin{align*}
&\varphi_{0\delta}:=(I+\sqrt{\delta} B)^{-1}\,\varphi_0.
\end{align*}
Then we have $\varphi_{0\delta}+\sqrt{\delta}B\varphi_{0\delta}=\varphi_0$ and so
$\overline{\varphi}_{0\delta}+\sqrt{\delta}\overline{\varphi}_{0\delta}=\overline{\varphi}_{0}$, whence
\begin{align}
&|\overline{\varphi}_{0\delta}|=\frac{|\overline{\varphi}_{0}|}{1+\sqrt{\delta}}<|\overline{\varphi}_{0}|<1.
\label{avcont}
\end{align}
Moreover, by introducing the convex function $G$ defined as in \eqref{Gdef}
\begin{align*}
&G(s)=F(s)+\frac{a_\infty}{2}s^2,
\end{align*}
and by multiplying the relation
$\varphi_{0\delta}-\varphi_0=-\sqrt{\delta}B\varphi_{0\delta}$ by $G'(\varphi_{0\delta})$ in $L^2(\Omega)$
we obtain
\begin{align*}
&\int_\Omega G'(\varphi_{0\delta})(\varphi_{0\delta}-\varphi_0)=-\sqrt{\delta}\int_{\Omega} G''(\varphi_{0\delta})|\nabla\varphi_{0\delta}|^2
-\sqrt{\delta}\int_\Omega G'(\varphi_{0\delta})\varphi_{0\delta}\leq 0,
\end{align*}
owing to the fact that $G'$ is monotone nondecreasing and that it is not restrictive to assume that $F'(0)=0$ (and hence $G'(0)=0$).
Therefore, the convexity of $G$ yields
\begin{align}
&\int_\Omega G(\varphi_{0\delta})\leq\int_\Omega G(\varphi_{0})+\int_\Omega G'(\varphi_{0\delta})(\varphi_{0\delta}-\varphi_0)
\leq \int_\Omega G(\varphi_{0}),\label{convexG}
\end{align}
and, since $F(\varphi_0)\in L^1(\Omega)$, from this last inequality we deduce that $F(\varphi_{0\delta})\in L^1(\Omega)$
and hence that $\varphi_{0\delta}\in L^\infty(\Omega)$ with $|\varphi_{0\delta}|\leq 1$ a.e. in $\Omega$.
In order to deduce the first of $\eqref{auxlem}$, notice that we have
\begin{align*}
&\Vert\varphi_{0\delta}\Vert^2+\sqrt{\delta}\Vert\varphi_{0\delta}\Vert_V^2=(\varphi_0,\varphi_{0\delta})
\leq\frac{1}{2}\Vert\varphi_0\Vert^2+\frac{1}{2}\Vert\varphi_{0\delta}\Vert^2,
\end{align*}
whence
\begin{align*}
&\delta\Vert\nabla\varphi_{0\delta}\Vert^2\leq\frac{\sqrt{\delta}}{2}\Vert\varphi_{0}\Vert^2\to 0,\qquad\mbox{as }\delta\to 0.
\end{align*}
Finally, the convergence in the second of \eqref{auxlem} follows immediately from the theory of maximal monotone operators.
\end{proof}

Now, choosing the initial data $\varphi_{0\delta}$ as given by Lemma \ref{lemaux}, thanks to Lemma \ref{existence2} we know that for every $\delta>0$ Problem P$_\delta$ admits a
weak solution $[\uvec_\delta,\varphi_\delta]$ having the regularity properties \eqref{regp1Pb1}--\eqref{pointbdphi}
and satisfying the energy inequality
\begin{align}
&\int_\Omega\frac{1}{2}\rho(\varphi_\delta)\uvec_\delta^2+E(\varphi_\delta)+\frac{\delta}{2}\Vert\nabla\varphi_\delta\Vert^2+
2\int_0^t\Vert\sqrt{\nu(\varphi_\delta)}D\uvec_\delta\Vert^2 d\tau+\delta\int_0^t\Vert A^{3/2}\uvec_\delta\Vert^2 d\tau\nonumber\\
&+\int_0^t\big\Vert\sqrt{m(\varphi_\delta)}\nabla\mu_\delta\big\Vert^2 d\tau\leq\int_\Omega\frac{1}{2}\rho(\varphi_{0\delta})\uvec_0^2
+E(\varphi_{0\delta})+\frac{\delta}{2}\Vert\nabla\varphi_{0\delta}\Vert^2
+\int_0^t\langle\hvec,\uvec_\delta\rangle_{V_{div}} d\tau,
\label{enindelta}
\end{align}
for almost all $t\in(0,T)$.
Since $\delta$ is no longer fixed here, the uniform with respect to $\delta$ estimates that we shall be able to deduce
for the sequence of $[\uvec_\delta,\varphi_\delta]$
will be weaker than the estimates obtained in Step I and Step II. 
Nevertheless, these estimates will turn out to be enough to pass to the limit in P$_\delta$.
More precisely, noting that the right hand side of \eqref{enindelta} is bounded due to Lemma \ref{lemaux}
(see the first of \eqref{auxlem} and \eqref{convexG}),
the only uniform in $\delta$ estimates that we can write from \eqref{enindelta} are now
\begin{align}
&\Vert\uvec_\delta\Vert_{L^\infty(0,T;G_{div})\cap L^2(0,T;V_{div})}\leq C,\label{bd1delta}\\
&\Vert\varphi_\delta\Vert_{L^\infty(0,T;L^p(\Omega))\cap L^2(0,T;V)}\leq C,\label{bd2delta}\\
&\Vert\nabla\mu_\delta\Vert_{L^2(0,T;H)}\leq C ,\label{bd3delta}\\
&\Vert F(\varphi_\delta)\Vert_{L^\infty(0,T;L^1(\Omega))}\leq C.\label{bd4delta}
\end{align}
Here, we have used (A1)--(A3), the bound $|\varphi_\delta|<1$ which yields
 $\rho(\varphi_\delta)>\widetilde{\rho}_\ast:=\min(\widetilde{\rho}_1,\widetilde{\rho}_2)>0$,
 and the fact that, as a consequence of \eqref{auxlem1} (taking $\e\to 0$) we have
\begin{align*}
&F(s)\geq C_p|s|^p-D_p,\qquad\forall s\in(-1,1),
\end{align*}
with $p\geq 3$ fixed arbitrary.

We now need to control the sequence of averages $\{\overline{\mu}_\delta\}$. To this aim we
first consider equation \eqref{Pb13} written in the form
\begin{align*}
&\varphi_\delta'+\uvec_\delta\cdot\nabla\varphi_\delta=-\mathcal{B}_{\varphi_\delta}\mu_\delta
\end{align*}
and test it by $\mathcal{N}_{\varphi_\delta}\big(F'(\varphi_\delta)-\overline{F'(\varphi_\delta)}\big)$
(recall that $F'(\varphi_\delta)\in L^2(0,T;V)$). Arguing as above we still get
\begin{align*}
&\Vert F'(\phid)-\overline{F'(\phid)}\Vert\leq C(\Vert\mathcal{N}_{\phid}\phid'\Vert
+\Vert\mathcal{N}_{\phid}(\ud\cdot\nabla\phid)\Vert+\sqrt{\delta}\Vert\phid\Vert_V)\nonumber\\
&\leq C(\Vert\phid'\Vert_{V'}+\Vert\ud\cdot\nabla\phid\Vert_{V'}+1),
\end{align*}
where we have used the bound $\delta^{1/2}\Vert\phid\Vert_{L^\infty(0,T;V)}\leq C$, which comes
from the energy inequality \eqref{enindelta}. Moreover, on account of the estimate
\begin{align}
&\Vert\varphi_\delta'\Vert_{L^2(0,T;V')}\leq C,\label{bd8delta}
\end{align}
which can be obtained by arguing exactly as in Step II by comparison in the
weak formulation of \eqref{Pb13},
we are still led to an estimate of the form
\begin{align}
&\big\Vert F'(\varphi_\delta)-\overline{F'(\varphi_\delta)}\big\Vert_{L^2(0,T;H)}\leq C.
\label{Ken1}
\end{align}
With this last estimate available we can now
 apply
the argument devised by Kenmochi et al. \cite{KNP} to deduce a bound of
$F'(\phid)$ in $L^2(0,T;L^1(\Omega))$. Let us recall some details
of this argument. For convenience, here we referee to, e.g., \cite[Proof of Theorem 1]{FG2}.
Introduce first the function
\begin{align*}
&H(s):=F(s)+\frac{a_\infty}{2}(s-s_0)^2,\qquad\forall s\in(-1,1),
\end{align*}
where $s_0\in(-1,1)$ is such that $F'(s_0)=0$ (cf. (A8)). Hence, owing to (A7) $H'$ is monotone
and $H'(s_0)=0$. Then, using the fact that $|\overline{\varphi}_{0\delta}|<1$
and exploiting the argument of Kenmochi et al.,
the following estimate can be established
\begin{align}
&\eta_\delta\Vert H'(\varphi_\delta)\Vert_{L^1(\Omega)}
\leq\int_\Omega(\varphi_\delta-\overline{\varphi}_{0\delta})\big(H'(\varphi_\delta)-\overline{H'(\varphi_\delta)}\big)
+K(\overline{\varphi}_{0\delta}),\label{Ken2}
\end{align}
where
$$\eta_\delta:=\min\{\overline{\varphi}_{0\delta}-m_1,m_2-\overline{\varphi}_{0\delta}\},\qquad
K(\overline{\varphi}_{0\delta})=(\eta_\delta+\kappa_\delta)|\Omega|\big(\max_{[m_1,m_2]}(|F_1'|+|F_2'|)+a_\infty\sigma\big),$$
with $\kappa_\delta:=\max\{\overline{\varphi}_{0\delta}-m_1,m_2-\overline{\varphi}_{0\delta}\}$,
$\sigma:=\max\{s_0-m_1,m_2-s_0\}$, and $m_1,m_2\in (-1,1)$ are two constants that are independent of $\delta$ and fixed
such that $m_1\leq s_0\leq m_2$ and $m_1<\overline{\varphi}_{0\delta}<m_2$ for all $\delta>0$.
Due to \eqref{avcont} we see that $m_1,m_2$ can be fixed in this way. Indeed, it is enough
to fix $-1<m_1<\min\{-|\overline{\varphi}_0|,s_0\}$ and $\max\{|\overline{\varphi}_0|,s_0\}< m_2<1 $.
Moreover we can see also that the constants $\eta_\delta$, $\kappa_\delta$ and hence $K(\overline{\varphi}_{0\delta})$
are all uniformly bounded (from below and above) with respect to $\delta$.
In particular, we have $\kappa_\delta< 1$ and $\min\{-|\overline{\varphi}_0|-m_1,m_2-|\overline{\varphi}_0|\}<\eta_\delta< 1$, for all $\delta$.
Hence, the constants $K(\overline{\varphi}_{0\delta})$ are bounded by a constant 
which depends only on $\overline{\varphi}_0$ (and on $F$, $J$ and $\Omega$).

By combining \eqref{Ken1} (which holds also with $H$ in place of $F$) with \eqref{Ken2} we deduce the desired bound
\begin{align*}
&\big\Vert F'(\varphi_\delta)\big\Vert_{L^2(0,T;L^1(\Omega))}\leq
L\big(\overline{\varphi}_0,\mathcal{E}(\uvec_0,\varphi_0),\Vert\hvec\Vert_{L^2(0,T;V_{div}')}\big),
\end{align*}
and this provides the control $\Vert\overline{\mu}_\delta\Vert_{L^2(0,T)}\leq L$, as well as the control
\begin{align}
&\Vert\mu_\delta\Vert_{L^2(0,T;V)}\leq C,\label{bd5delta}
\end{align}
with a constant $C$ now depending also on $\overline{\varphi}_0$, which is
derived by using \eqref{bd3delta} and Poincar\'{e}-Wirtinger inequality.

Let us now deduce an estimate for the sequence of time derivatives
$(\rho(\varphi_\delta)\uvec_\delta)_t$. Recalling the weak formulation of \eqref{Pb11}, for every $\wvec\in D(A^{3/2})$ we have
\begin{align}
&\big\langle(\rho(\varphi_\delta)\uvec_\delta)_t,\wvec\big\rangle_{D(A^{3/2})}
=
\big(\rho(\varphi_\delta)\uvec_\delta\otimes \uvec_\delta,D\wvec\big)
-2\big(\nu(\phid)D\uvec_\delta,D\wvec\big)-\delta(A^{3/2}\uvec_\delta,A^{3/2}\wvec)\nonumber\\
&
+\int_\Omega \uvec_\delta\cdot (\Jvect_\delta\cdot\nabla)\wvec
-(\varphi_\delta\nabla\mu_\delta,\wvec)+\langle\hvec,\wvec\rangle_{V_{div}},\label{varfor1}
\end{align}
where
\begin{align*}
&\Jvect_\delta=-\beta m(\varphi_\delta)\nabla\mu_\delta.
\end{align*}
From \eqref{varfor1}, on account of \eqref{bd1delta}, \eqref{bd3delta} and of the bound
$\delta^{1/2}\Vert\ud\Vert_{L^2(0,T;D(A^{3/2}))}\leq C$ which comes from
the energy inequality \eqref{enindelta}, it is easy to deduce
\begin{align}
&\Vert(\rho(\varphi_\delta)\uvec_\delta)_t\Vert_{L^2(0,T;D(A^{3/2})')}\leq C.\label{bd6delta}
\end{align}
The uniform in $\delta$ bounds \eqref{bd1delta}, \eqref{bd2delta} and
\eqref{bd5delta} imply the existence of $\uvec$, $\varphi$ and $\mu$ such that up to a subsequence we have
\begin{align}
&\uvec_\delta\rightharpoonup \uvec,\qquad\mbox{weakly$^\ast$ in } L^\infty(0,T;G_{div}),\label{weak1delta}\\
&\uvec_\delta\rightharpoonup \uvec,\qquad\mbox{weakly in } L^2(0,T;V_{div}),\label{weak2delta}\\
&\varphi_\delta\rightharpoonup \varphi,\qquad\mbox{weakly$^\ast$ in } L^\infty(0,T; L^p(\Omega)),\label{weak3delta}\\
&\varphi_\delta\rightharpoonup \varphi,\qquad\mbox{weakly in } L^2(0,T;V),\label{weak4delta}\\
&\mu_\delta\rightharpoonup \mu,\qquad\mbox{weakly in } L^2(0,T;V).\label{weak5delta}
\end{align}

Now, we have
$$\nabla(\rho(\varphi_\delta)\uvec_\delta)=\rho(\varphi_\delta)\nabla\uvec_\delta
+\rho'(\varphi_\delta)\nabla\varphi_\delta\cdot\uvec_\delta,$$
and since $\uvec_\delta$ is bounded in $L^\infty(0,T;L^2(\Omega)^3)\cap L^2(0,T;L^6(\Omega)^3)$
and hence also in $L^{10/3}(Q)^3$, then we see that
$\rho'(\varphi_\delta)\nabla\varphi_\delta\cdot\uvec_\delta$ is bounded in $L^{5/4}(Q)$ which implies the bound
\begin{align}
&\Vert\rho(\varphi_\delta)\uvec_\delta\Vert_{L^{5/4}(0,T;W^{1,5/4}(\Omega)^3)}\leq C.\label{bd7delta}
\end{align}
By exploiting the compact embedding $W^{1,5/4}(\Omega)^3\hookrightarrow\hookrightarrow L^{15/7-\gamma}(\Omega)^3$
 and the fact that we have $L^{15/7-\gamma}(\Omega)^3\hookrightarrow D(A^{3/2})'$,
since $D(A^{3/2})\hookrightarrow L^{15/8+\gamma '}(\Omega)^3$ (here $\gamma,\gamma'>0$),
by Aubin-Lions lemma from \eqref{bd6delta} and \eqref{bd7delta} we infer that
\begin{align}
&\rho(\varphi_\delta)\uvec_\delta\to\overline{\rho(\varphi)\uvec},\qquad\mbox{strongly in } L^{5/4}(0,T;L^{15/7-\gamma}(\Omega)^3),\qquad
\gamma>0.
\label{strong1delta}
\end{align}
Now, observe that \eqref{bd8delta}
and \eqref{bd2delta} imply that up to a subsequence we have
\begin{align}
&\varphi_\delta\to\varphi,\qquad\mbox{strongly in } L^2(Q),\mbox{ and pointwise a.e. in }Q,\label{strong2delta}\\
&\rho(\phid)\to\rho(\varphi),\qquad\mbox{strongly in } L^q(Q),\quad 2\leq\forall q<\infty.
\end{align}
Hence, since $\uvec_\delta\rightharpoonup\uvec$ weakly in $L^{10/3}(Q)^3$, then we have
$\rho(\varphi_\delta)\uvec_\delta\rightharpoonup\rho(\varphi)\uvec$ weakly in $L^{10/3-\gamma}(Q)^3$
($\gamma>0$)
and by comparison with \eqref{strong1delta} we get
$\overline{\rho(\varphi)\uvec}=\rho(\varphi)\uvec$.
Hence, from \eqref{bd6delta} and \eqref{bd8delta} we have
\begin{align}
&(\rho(\varphi_\delta)\uvec_\delta)_t\rightharpoonup(\rho(\varphi)\uvec)_t,\qquad\mbox{weakly in }L^2(0,T;D(A^{3/2})'),\\
&\varphi_\delta'\rightharpoonup\varphi_t,\qquad\mbox{weakly in }L^2(0,T;V').
\end{align}
We are now ready to pass to the limit in the variational formulation of Problem P$_\delta$
 and therefore to prove that the original problem \eqref{Pb21}--\eqref{Pb27} admits a weak solution.
To this goal, observe that, due to the interpolation embedding
\begin{align*}
&L^\infty(0,T;G_{div})\cap L^2(0,T;V_{div})\hookrightarrow L^r(0,T;L^{6r/(3r-4)}(\Omega)^3),\qquad 2\leq r\leq \infty,
\end{align*}
we have the bound for the sequence of $\uvec_\delta$ in $L^5(0,T;L^{30/11}(\Omega)^3)$
and hence
\begin{align}
&\uvec_\delta\rightharpoonup\uvec,\qquad\mbox{weakly in }L^{5}(0,T;L^{30/11}(\Omega)^3).\label{weak6delta}
\end{align}
The strong and weak convergences  \eqref{strong1delta} and \eqref{weak6delta}, respectively, are enough to pass to the limit
in the term $\mbox{div}(\rho(\varphi_\delta)\uvec_\delta\otimes\uvec_\delta)$ (cf. \eqref{varfor1}; recall that the test function
$\wvec\in D(A^{3/2})\hookrightarrow H^3(\Omega)^3$ and hence $D\wvec\in L^\infty(\Omega)^{3\times 3}$).

 Furthermore, by employing the pointwise convergence $\uvec_\delta\to\uvec$ a.e. in $Q$,
which follows from the strong convergence \eqref{strong1delta} and from
the pointwise convergence \eqref{strong2delta} for $\varphi_\delta$, we deduce the strong convergence
\begin{align}
&\uvec_\delta\to\uvec,\qquad\mbox{strongly in }L^2(0,T;G_{div})\label{strong3delta}
\end{align}
(actually the convergence of $\ud$ to $\uvec$ is strong also in $L^{10/3-\gamma}(Q)^3$).
On the other hand, since $m(\varphi_\delta)\to m(\varphi)$ strongly in $L^q(Q)$ for all $q<\infty$, by using
\eqref{weak5delta} and the bound $\Vert\Jvect_\delta\Vert_{L^2(Q)^2}\leq C$, we obtain
\begin{align}
&\Jvect_\delta\rightharpoonup\Jvect:=-\beta m(\varphi)\nabla\mu,\qquad\mbox{weakly in }L^2(Q)^3.
\label{weak7delta}
\end{align}
The strong and weak convergences \eqref{strong3delta} and \eqref{weak7delta} are now enough to pass to the limit
in the (variational formulation of the) term $\mbox{div}(\uvec_\delta\otimes\Jvect_\delta)$.

As far as the term $\delta A^3\uvec_\delta$ is concerned, we observe that from the energy inequality
\eqref{enindelta} we have the bound $\delta^{1/2}\Vert A^{3/2}\uvec_\delta\Vert_{L^2(0,T;L^2(\Omega)^3)}\leq C$
and therefore the contribution of this term in the variational formulation of P$_\delta$ vanishes as $\delta\to 0$.
Indeed, recalling the standard argument to pass to the limit in \eqref{varfor1} by
first multiplying it by a test function $\chi\in C^\infty_0(0,T)$ and then integrate the resulting
identity in time between 0 and $T$, we have,
for all $\wvec\in D(A^{3/2})$
\begin{align*}
&\delta\Big|\int_0^T(A^{3/2}\uvec_\delta,A^{3/2}\wvec)\chi(t)dt\Big|\leq \sqrt{\delta} C\Vert A^{3/2}\wvec\Vert\to 0,\qquad\mbox{as }\delta\to 0.
\end{align*}
Regarding the term $-\delta\Delta\varphi_\delta$, we see that also the contribution of this term vanishes as $\delta\to 0$.
Indeed,
from \eqref{enindelta} we have the bound $\delta^{1/2}\Vert\varphi_\delta\Vert_{L^\infty(0,T;V)}\leq C$, and therefore for every
test functions $\zeta\in V$ and $\chi\in C^\infty_0(0,T)$ we have
\begin{align}
&\delta\Big|\int_0^T(\nabla\varphi_\delta,\nabla\zeta)\chi(t)dt\Big|\leq\delta^{1/2}C\Vert\nabla\zeta\Vert\to 0,\qquad\mbox{as }\delta\to 0.
\label{passl}
\end{align}
Now, from the variational formulation of \eqref{Pb14} we have
\begin{align}
&\int_0^T(\mu_\delta,\zeta)\chi(t)dt=\int_0^T(a\varphi_\delta-J\ast\varphi_\delta+F'(\varphi_\delta),\zeta)\chi(t)dt
+\delta\int_0^T(\nabla\varphi_\delta,\nabla\zeta)\chi(t)dt,\label{wfchpotdelta}
\end{align}
for all $\zeta\in V$ and all $\chi\in C^\infty_0(0,T)$.
In order to pass to the limit in the term containing $F'(\phid)$ of \eqref{wfchpotdelta}
 we need to show that $|\varphi|<1$ a.e. in $Q$
(observe that up to now, from the pointwise convergence \eqref{strong2delta} and
the strict bound $|\phid|<1$ for all $\delta$ we only know that
$|\varphi|\leq 1$ a.e. in $Q$). To this purpose, we can employ
exactly the same argument as in \cite[Proof of Theorem 1]{FG2},
using the pointwise convergence $\phid\to\varphi$, the bound $\Vert F'(\phid)\Vert_{L^1(Q)}\leq C(\overline{\varphi}_0)$,
and assumption (A8).

Therefore, since $F'$ is continuous on $(-1,1)$, we have $F'(\varphi_\delta)\to F'(\varphi)$ pointwise in $Q$.
This pointwise convergence, together with the bound $\Vert F'(\phid)\Vert_{L^2(0,T;H)}\leq C$
(which follows from \eqref{Ken1} and from the bound
of $\overline{\mu}_\delta$ in $L^2(0,T)$)
yield, up to a subsequence
\begin{align}
&F'(\phid)\rightharpoonup F'(\varphi),\qquad\mbox{weakly in }L^2(0,T;H).\label{passlw}
\end{align}
 We can then pass to the limit as $\delta\to 0$ in \eqref{wfchpotdelta},
using \eqref{passl} and the weak convergences \eqref{passlw}, \eqref{weak5delta}
This shows that \eqref{Pb24} is satisfied by the limit functions $\varphi$ and $\mu$.


The argument for passing to the limit in the other terms of Problem P$_\delta$ is
straightforward and therefore, by letting $\delta\to 0$ from the weak formulation
of Problem P$_\delta$ we finally recover the weak formulation of the original problem \eqref{Pb21}--\eqref{Pb27}
with test functions $\wvec\in D(A^{3/2})$ and $\psi\in V$ for \eqref{Pb21} and \eqref{Pb23}, respectively.

We can now observe that, by density, the weak formulation of \eqref{Pb21} holds also
for every test function $\wvec\in D(A)$.
Moreover, for all $\wvec\in D(A)$ we have
\begin{align*}
\Big|\int_\Omega(\uvec\otimes\Jvect):D\wvec\Big|&\leq C\Vert\uvec\Vert_{L^3(\Omega)}\Vert\Jvect\Vert\Vert\wvec\Vert_{D(A)}
\leq C\Vert\uvec\Vert^{1/2}\Vert\uvec\Vert_{L^6(\Omega)^3}^{1/2}\Vert\Jvect\Vert\Vert\wvec\Vert_{D(A)},\nonumber\\
&\leq C\Vert\uvec\Vert^{1/2}\Vert\nabla\uvec\Vert^{1/2}\Vert\Jvect\Vert\Vert\wvec\Vert_{D(A)}
\end{align*}
whence we obtain
\begin{align*}
&\mbox{div}(\uvec\otimes\Jvect)\in L^{4/3}(0,T;D(A)').
\end{align*}
Indeed, we have
\begin{align*}
\Big|\int_\Omega(\rho\uvec\otimes\uvec):D\wvec\Big|&\leq C\Vert\uvec\Vert_{L^4(\Omega)^3}^2\Vert\wvec\Vert_{V_{div}}
\leq C\Vert\uvec\Vert^{1/2}\Vert\uvec\Vert_{L^6(\Omega)^3}^{3/2}\Vert\wvec\Vert_{V_{div}}\nonumber\\
&\leq C\Vert\uvec\Vert^{1/2}\Vert\nabla\uvec\Vert^{3/2}\Vert\wvec\Vert_{V_{div}},
\end{align*}
for all $\wvec\in V_{div}$, which entails
\begin{align*}
&\mbox{div}(\rho\uvec\otimes\uvec)\in L^{4/3}(0,T;V_{div}').
\end{align*}
It is also immediate to check that the terms $-2\mbox{div}\big(\nu(\varphi)D\uvec\big)$ and $\mu\nabla\varphi$ in \eqref{Pb21} are in
$L^2(0,T;V_{div}')$. Hence, \eqref{weakfor1} is satisfied for all $\wvec\in D(A)$ and we have
\begin{align*}
&(\rho\uvec)_t\in L^{4/3}(0,T;D(A)').
\end{align*}

Let us now prove the weak continuity $\uvec\in C_w([0,T];G_{div})$. To this purpose note that
from \eqref{bd1delta} and \eqref{bd6delta} we have the boundedness of
\begin{align}
&\rho(\phid)\ud\quad\mbox{ in } L^\infty(0,T;L^2(\Omega)^3)\quad\mbox{and}\nonumber\\
&\rho(\phid)\ud\quad\mbox{  in } H^1(0,T;D(A^{3/2})')\hookrightarrow C([0,T];D(A^{3/2})').\nonumber
\end{align}
Therefore, Lemma \ref{Strauss} and \eqref{strong1delta} imply that $\rho(\varphi)\uvec\in C_w([0,T];L^2(\Omega)^3)$
and this leads to the weak continuity of $\uvec$ in $G_{div}$ as in Step I, on account of the strong continuity
$\varphi\in C([0,T];H)$.

The argument to prove that $\uvec$ and $\varphi$ attain the initial values $\uvec_0$ and $\varphi_0$, respectively,
follows the same lines as in Step I and we omit it.

Finally we prove the energy inequality \eqref{eninPbor}.
First, we know from Step II that $[\ud,\phid]$ satisfies \eqref{eninomega}, namely
\begin{align}
&\Big(\int_\Omega\frac{1}{2}\rho(\varphi_{0\delta})\uvec_{0}^2+E(\varphi_{0\delta})+\frac{\delta}{2}\Vert\nabla\varphi_{0\delta}\Vert^2\Big)\omega(0)+
\int_0^T\Big(\int_\Omega\frac{1}{2}\rho(\phid)\ud^2+E(\phid)+\frac{\delta}{2}\Vert\nabla\phid\Vert^2\Big)\omega'(\tau)d\tau\nonumber\\
&\geq\int_0^T\Big(2\Vert\sqrt{\nu(\phid)}D\ud\Vert^2+\delta\Vert A^{3/2}\ud\Vert^2
+\Vert\sqrt{m(\phid)}\nabla\mud\Vert^2\Big)\omega(\tau)d\tau-\int_0^T\langle\hvec,\ud\rangle_{V_{div}}
\omega(\tau)d\tau,
\label{eninomegadelta}
\end{align}
for every $\omega\in W^{1,1}(0,T)$, with $\omega(T)=0$ and $\omega\geq 0$.
Again, to pass to the limit in \eqref{eninomegadelta} we need to show that
\begin{align}
&\int_0^T\omega'(\tau)d\tau\int_\Omega F(\phid)\to\int_0^T\omega'(\tau)d\tau\int_\Omega F(\varphi).\label{claim3}
\end{align}
This convergence can be established by the same argument as at the end of Step II, by exploiting the
fact that $F$ is a quadratic perturbation of a convex function and by using the bound of $F'(\phid)$ in
$L^2(0,T;H)$. Note that if $F$ is bounded, then \eqref{claim3} follows at once by directly
applying Lebesgue's theorem.
Concerning the other terms in \eqref{eninomegadelta}, we have
\begin{align}
&\int_0^T\omega'(\tau)d\tau\int_\Omega\frac{1}{2}\rho(\phid)\ud^2\to\int_0^T\omega'(\tau)d\tau\int_\Omega\frac{1}{2}\rho(\varphi)\uvec^2,
\label{claim4}
\end{align}
thanks to the strong convergences $\ud\to\uvec$ in $L^3(Q)^3$
(which follows from the bound of $\ud$ in $L^{10/3}(Q)^3$ and pointwise convergence)
 and $\rho(\phid)\to\rho(\varphi)$ in $L^q(Q)$ for all $q<\infty$
(take $q=3$). Convergence \eqref{claim4} can also be justified by observing that
we have also $\int_\Omega\frac{1}{2}\rho(\phid)\ud^2\to\int_\Omega\frac{1}{2}\rho(\varphi)\uvec^2$ for almost
all $\tau\in(0,T)$ and by applying Lebesgue's theorem.

Furthermore, we have
\begin{align*}
&\int_\Omega\frac{1}{2}\rho(\varphi_{0\delta})\uvec_0^2\to\int_\Omega\frac{1}{2}\rho(\varphi_{0})\uvec_0^2,
\end{align*}
thanks to Lebesgue's theorem,
\begin{align*}
&\int_\Omega F(\varphi_{0\delta})\leq\int_\Omega F(\varphi_0)
+\frac{a_\infty}{2}\big(\Vert\varphi_{0}\Vert^2-\Vert\varphi_{0\delta}\Vert^2\big)
\to\int_\Omega F(\varphi_0),\qquad\mbox{as }\delta\to 0,
\end{align*}
thanks to \eqref{convexG} and to the second of \eqref{auxlem}, and
\begin{align*}
&\Big|\int_0^T\frac{\delta}{2}\Vert\nabla\phid\Vert^2\omega'(\tau)d\tau\Big|\leq C_\omega\Vert\phid\Vert_{L^2(0,T;V)}^2
\leq C\delta\to 0,\qquad\mbox{as }\delta\to 0,
\end{align*}
thanks to the bound \eqref{bd2delta}.
By also using the first of \eqref{auxlem} and weak lower semicontinuity of norms (the second term on the right
hand side of the inequality is simply neglected) we can now pass to the limit in \eqref{eninomegadelta}
and obtain the following integral inequality
\begin{align}
&\mathcal{E}(0)\omega(0)+
\int_0^T\mathcal{E}(\tau)\omega'(\tau)d\tau
\geq\int_0^T\mathcal{D}(\tau)d\tau,\label{auxineq}
\end{align}
satisfied for every $\omega\in W^{1,1}(0,T)$, with $\omega(T)=0$ and $\omega\geq 0$,
where the functions $\mathcal{E}$ and $\mathcal{D}$ are given by
\begin{align*}
&\mathcal{E}(t):=\int_\Omega\frac{1}{2}\rho(\varphi)\uvec^2+E(\varphi),
\qquad\mathcal{D}(t):=2\Vert\sqrt{\nu(\varphi)}D\uvec\Vert^2
+\Vert\sqrt{m(\varphi)}\nabla\mu\Vert^2-\langle\hvec,\uvec\rangle_{V_{div}},
\end{align*}
for all $t\in[0,T]$ and almost all $t\in(0,T)$, respectively (for simplicity of notation, we omit
the indication of time $t$ on the right hand side).
Let us check that $\mathcal{E}=\mathcal{E}(\uvec(\cdot),\varphi(\cdot)):[0,\infty)\to\mathbb{R}$ is lower semicontinuous.
Indeed, we know that $\rho(\varphi)\uvec\in C_w([0,T];L^2(\Omega)^3)$ and moreover we have
$\rho(\varphi)^{-1/2}\in C([0,T];H)$. Therefore $\sqrt{\rho(\varphi)}\uvec
=\rho(\varphi)^{-1/2}\rho(\varphi)\uvec\in C_w([0,T];L^2(\Omega)^3)$, on account of the boundedness
of $\rho$ and this proves the lower semicontinuity in time of the first term of $\mathcal{E}$.
The lower semicontinuity of $E(\varphi(\cdot)):[0,\infty)\to\mathbb{R}$
is a consequence of the fact that $F$ is a quadratic perturbation of a convex function (see \cite[Lemma 2]{FG1}).

The energy inequality \eqref{eninPbor} now follows by applying Lemma \ref{Abelslemma} to \eqref{auxineq}.

The proof of Theorem \ref{mainres} is now complete.

\hspace{160mm}$\Box$

\begin{oss}\label{2Dcase}
{\upshape \textbf{The two dimensional case.}
The proof of Theorem \ref{mainres} can be obviously carried out
in two space dimensions as well. In particular, the strong convergences
that can be obtained for the sequences of approximate solutions at each step
by using interpolation and Aubin-Lions lemma will generally hold in
stronger norms with respect to the 3D case. Nevertheless, after passing to the limit as $\delta\to 0$
the weak solution we get for system \eqref{Pbor1}--\eqref{Pbor8} still has
no more than the regularity given by \eqref{rewsol1}-\eqref{rewsol3}, \eqref{rewsol5} and
by the second of \eqref{rewsol4}. Only the regularity for $(\rho\uvec)_t$ improves a bit.
Indeed, by a comparison in the weak formulation \eqref{weakfor1} with test function
$\wvec\in D(A^{s/2})$, for $1<s\leq 2$, using interpolation and Gagliardo-Nirenberg
inequality in 2D it is not difficult to see that
\begin{align}
&(\rho\uvec)_t\in L^{2-\gamma}(0,T;D(A)')\cap L^{2/(3-r)}(0,T;D(A^{r/2})'),\nonumber
\end{align}
for every $0<\gamma\leq 1$ and every $1<r<2$. Notice that we cannot set $r=1$.
Indeed, in the weak formulation
\eqref{weakfor1} we cannot take the test function $\wvec$ in $V_{div}$
due to the extra-term $\mbox{div}(\uvec\otimes\Jvect)$ and to the
fact that $\uvec$ does not belong to $L^\infty(\Omega)^2$.
Therefore, even if the regularity for $(\rho\uvec)_t$
slightly improves in 2D, this is not enough to show
the validity of the energy identity. For the same motivations
also uniqueness of weak solutions in 2D is not known.

In conclusion, in order to get an improvement of the results in 2D concerning
uniqueness and validity of the energy identity, we first need
to prove a regularity result for system \eqref{Pbor1}--\eqref{Pbor8},
namely, to establish existence of solutions that are more regular than the
ones constructed in Theorem \ref{mainres} (a regularity assumption on the initial
data will then be required). This will be the subject of a future contribution.
}

\end{oss}

\vspace{5mm}

{\bf Acknowledgement.} I would like to thank Helmut Abels for a fruitful discussion concerning one point of the proof of
Lemma \ref{existence1}.  This work  was supported by the
 FP7-IDEAS-ERC-StG \#256872 (EntroPhase).
The author is a member of GNAMPA (Gruppo Nazionale per l'Analisi Matematica, la Probabilit\`a e le loro Applicazioni) of INdAM (Istituto Nazionale di Alta Mate-\\matica).

\end{document}